\documentclass[twoside,11pt]{article} 
\usepackage{jmlr2e_cmTECHREP}

\usepackage{bm}

\usepackage[backend=biber,style=numeric, firstinits=true, maxalphanames=4, maxnames=99,minnames=99]{biblatex}
\renewbibmacro{in:}{}
\DeclareSourcemap{
  \maps[datatype=bibtex]{
    \map{
      \step[fieldsource=journal, match={Numerische Mathematik}, replace={Numer. Math.} ]
    }
     \map{
      \step[fieldsource=journal, 
     match={SIAM Journal on Numerical Analysis}, replace={SIAM J. Numer. Anal.} ]
    }
    \map{
      \step[fieldsource=journal, 
     match={Journal of Computational Physics}, replace={J. Comput. Phys.} ]
    } 
  }
}

\usepackage{breakurl}
\usepackage{mathrsfs }

\newcommand{\x }{\bm{x}}

\usepackage{amsmath,amssymb,amsxtra,comment,graphicx,psfrag}
\usepackage{bm,mathrsfs}
\usepackage{mathtools}
\usepackage{xspace}
\usepackage{stmaryrd}

\usepackage{enumerate,comment,graphicx,psfrag}

\usepackage{hhline}

\usepackage{bookmark}
\usepackage{nameref}

\usepackage{verbatim}
\usepackage{tabls}

\usepackage{fancyhdr}
\usepackage{epsfig}
\usepackage{epsfig,subfigure,epstopdf}

\usepackage{pstricks,pst-plot,pst-func}
\usepackage{pspicture}
\usepackage{curves}

\usepackage{bbm}

\include{rgb}

\def\nst2{\| _*} 
 
\def\a12{A_h ^{1/2} } 
\def\d{{\mathrm d}}
\def\tr|{|\! |\! |}

\def\R {{\mathbb R}}
\def\P{{\mathbb P}}
\def\N{{\mathbb N}}

\def\N{{\mathbb N}}
\def\E{{\mathcal{E}}}

\def\D{\mathscr{D}}

\def\AA{\text{${\mathcal O}\hskip-3pt\iota$}}

\def\wU{\widehat U}

\def \x{\overline x}

\def \a{\alpha }

\def\T_h{{{\mathcal T}_h}}

\def\<{{\langle }}
\def\>{{\rangle }}
\def\V{{\mathcal{V}}}

\def\MR#1{\href{http://www.ams.org/mathscinet-getitem?mr=#1}{MR#1}}

\def\nst2{\| _*} 
 
\def\a12{A_h ^{1/2} } 
\def\d{{\mathrm d}}
\def\tr|{|\! |\! |}

\def\R {{\mathbb R}}
\def\P{{\mathbb P}}
\def\N{{\mathbb N}}

\def\N{{\mathbb N}}
\def\E{{\mathcal{E}}}

\def\D{\mathscr{D}}
\def\H{{\mathcal{H}}}

\def\AA{\text{${\mathcal O}\hskip-3pt\iota$}}

\def\wU{\widehat U}

\def \a{\alpha }

\def\T_h{{{\mathcal T}_h}}

\def\<{{\langle }}
\def\>{{\rangle }}
\def\V{{\mathcal{V}}}

\DeclareSymbolFont{matha}{OML}{txmi}{m}{it}
\DeclareMathSymbol{\varv}{\mathbf}{matha}{118}

\def\D{\mathscr{D}}

\newcommand\restr[2]{{
  \left.\kern-\nulldelimiterspace 
  #1 
  \vphantom{\big|} 
  \right|_{#2} 
  }}

\NewDocumentCommand{\dgal}{sO{}m}{%
  \IfBooleanTF{#1}
    {\dgalext{#3}}
    {\dgalx[#2]{#3}}%
}

\NewDocumentCommand{\dgalext}{m}{%
  \sbox0{%
    \mathsurround=0pt 
    $\left\{\vphantom{#1}\right.\kern-\nulldelimiterspace$%
  }%
  \sbox2{\{}%
  \ifdim\ht0=\ht2
    \{\kern-.625\wd2 \{#1\}\kern-.625\wd2 \}%
  \else
    \left\{\kern-.7\wd0\left\{#1\right\}\kern-.7\wd0\right\}%
  \fi
}

\NewDocumentCommand{\dgalx}{om}{%
  \sbox0{\mathsurround=0pt$#1\{$}%
  \sbox2{\{}%
  \ifdim\ht0=\ht2
    \{\kern-.625\wd2 \{#2\}\kern-.625\wd2 \}%
  \else
    \mathopen{#1\{\kern-.7\wd0 #1\{}
    #2
    \mathclose{#1\}\kern-.7\wd0 #1\}}
  \fi
}

\renewcommand{\R}{{\mathbb{R}}}

\renewcommand{\T}{\mathbb{T}^d}

\ShortHeadings{}{}
\firstpageno{1}

\begin{document}


\title
{\sc Runge--Kutta Physics Informed Neural Networks:\\ Formulation and Analysis}

 \author{ \name{\begin{center}{Georgios Akrivis$^{1,3}$,   Charalambos G.\ Makridakis$^{2, 3, 4}$  \and  Costas Smaragdakis$^{ 3, 5}$ }\end{center}} 
         \addr{\begin{center}{ {$^{1}$ DCSE, University of Ioannina, Greece \\$^{2}$ DMAM, University of Crete, Greece \\
         $^{3 }$ IACM-FORTH, Greece
   }\\
    $^{4 }$ MPS, University of Sussex, United Kingdom\\
    $^{5}$ DSAFM, University of the Aegean, Greece } \end{center} }}


\editor{}
\maketitle
\abstract{In   this paper 
we consider time-dependent  PDEs discretized by a special class of  Physics Informed Neural Networks 
whose design is based on the framework of Runge--Kutta and related time-Galerkin discretizations. 
The primary motivation for using such methods is that alternative time-discrete schemes not only enable 
higher-order approximations but also have a crucial impact on the \emph{qualitative behavior} of the discrete solutions. 
The design of the methods  follows a novel training approach based on two key principles: 
(a) \emph {the discrete loss is designed using a time-discrete framework}, and 
(b) \emph {the final loss formulation incorporates Runge--Kutta or time-Galerkin discretization} in a carefully structured manner. 
We then demonstrate that the resulting methods inherit the stability properties of the Runge--Kutta 
or time-Galerkin schemes, and furthermore, their computational behavior aligns  with that of the original 
time discrete method used in their formulation. 
In our analysis, we focus on linear parabolic equations, demonstrating both the stability of the methods 
and the convergence of the discrete minimizers to solutions of the underlying evolution PDE. An important 
novel aspect of our work is the derivation of maximal regularity (MR) estimates for B-stable Runge--Kutta 
schemes and both continuous and discontinuous Galerkin time discretizations. This allows us to provide 
new energy-based proofs for maximal regularity estimates previously established by Kov\'acs, Li, and Lubich 
\cite{kovacs2016stable},  now in the Hilbert space setting and with the flexibility of variable time steps.}

\section{Introduction}\label{Se:1}

\subsection{Evolutionary PDEs and Neural Network Discretizations}\label{SSe:1.1}
  In   this paper, 
we consider time-dependent  PDEs discretized by a special class of  Physics Informed Neural Networks whose design is based 
on the framework of Runge--Kutta and related time-Galerkin discretizations. 
The key motivation for adopting such methods lies in their ability not only to achieve higher-order approximations but also to significantly influence the \emph{qualitative behavior} of the discrete solutions. These physically consistent approximations are critical for numerous applications, as underscored 
by the rich and significant 
literature
on time-discrete methods in numerical analysis and scientific computation;
 see, e.g., \cite{Butcher_1964}, \cite{Crouzeix_1}, \cite{Crouzeix_2},  \cite{BUTCHER1996113}, 
 \cite{Glow_theta1987}, \cite{HairerNW1993}, \cite{Hairer2002}, \cite{Thomee_book_2006}. Characteristic examples include 
 conservative schemes for wave and Schrödinger equations, dissipative 
 schemes for diffusion equations preserving the smoothing effect of the equation, geometry and structure preserving schemes, and specially tuned time discretizations 
 for the Navier--Stokes equations.


\subsubsection*{Formulation of the RK-PINN methods}
Physics Informed Neural Networks are algorithms where the discretization is based on the minimization of the $L^2$ 
norm of the residual of the evolution PDE over a set of neural networks with a given architecture. Computing the loss 
requires an additional quadrature step to evaluate the space-time integrals, a process known as \emph{training.} 
Standard approaches, particularly in high dimensions, often rely on probabilistic quadrature methods like Monte Carlo 
or Quasi-Monte Carlo. However, the impact of these discrete losses on the qualitative behavior of the approximations remains
 largely unexplored. Notably, even for the scalar wave equation, we lack clarity on whether and under what conditions 
 these methods preserve key conservation properties.

In this work, we propose a novel training approach based on two key principles: 

(a) \emph {the discrete loss is designed using a time-discrete framework}, and 

(b) \emph {the final loss formulation incorporates Runge--Kutta or time-Galerkin discretization} in a carefully structured manner.  

\noindent
We then demonstrate that the resulting methods inherit the stability properties of the Runge--Kutta or time-Galerkin 
schemes, and furthermore, their computational behavior aligns  with that of the original time discrete method used in their formulation. Since time-discrete training affects only the time variable, the resulting  schemes are particularly well-suited for high-dimensional evolution problems.

\subsubsection*{Analysis through novel Maximal Regularity estimates}
In our analysis, we focus on linear parabolic equations, demonstrating both the stability of the methods 
and the convergence of the discrete minimizers to solutions of the underlying evolution PDE. Following 
the approach in \cite{GGM_pinn_2023}, we employ the liminf-limsup framework of De Giorgi (see Section 2.3.4 
of \cite{DeGiorgi_sel_papers:2013} and \cite{braides2002gamma}), which is widely used in the 
$\varGamma$-convergence of functionals for nonlinear PDEs.

We first establish that the proposed methods yield stable functionals in the sense of properties [S1] and [S2], 
as defined in Section \ref{Se:3.1.1}. 
Our analysis reveals that stability is rooted in strong discrete regularity estimates associated with maximal regularity. 
A key novel contribution of our work is the derivation of maximal regularity (MR) estimates for B-stable Runge–Kutta 
schemes, as well as for both continuous and discontinuous Galerkin time discretizations.
This allows us to provide new energy-based proofs for maximal regularity estimates previously established 
by 
 \textcite{kovacs2016stable},  now in the Hilbert space setting and with the flexibility 
of variable time steps. This generalizes the results of \cite{kovacs2016stable}, where constant time steps 
were a key assumption. See also \cite{leykekhman2017discrete}, \cite{akrivis2022maximal}, \cite{KK_J_MaxmalReg_DG_2024}.
An additional interesting feature of our approach is the derivation of the first maximal regularity estimates for high-order Lobatto IIIA methods.

\subsubsection*{\it Plan of the paper}
In  Section \ref{Se:2} we systematically formulate the methods considered in this work. To motivate our approach, we use as a model a linear 
parabolic equation, but the algorithms are clearly applicable to a wide selection of evolution equations. 
First, we consider simple time discrete training methods and then we employ the pointwise formulation of \cite{AMN2} to design the various 
Runge--Kutta Physics Informed Neural Networks (RK-PINNs).    

Section \ref{Se:3}  is devoted to the analysis of the methods in the case of linear parabolic equations by establishing that the neural network 
approximations  satisfy crucial properties related to the  liminf-limsup framework mentioned above, \cite{DeGiorgi_sel_papers:2013}.
To prove stability, Proposition \ref{TD:EquicoercivityofG}, we rely on maximal regularity estimates. Our   novel  maximal regularity estimates 
for   Runge--Kutta   and both continuous and discontinuous Galerkin time discretizations are assumed in Section 3 and are systematically derived 
in Section \ref{Se:4}. Theorem \ref{Thrm:TimeD} demonstrates that the sequence of neural network  minimizers  converges 
to the exact solution of the parabolic equation, i.e., \begin{equation*}
\hat u_\ell \rightarrow u,  \;\; \; \textrm{in} \;\: L^2((0,T); H^1(\Omega))  ;
\end{equation*}
see  Theorem \ref{Thrm:TimeD}  for the precise statement. 
Throughout, we assume that the approximability capacity of the discrete neural network spaces is given; specifically, 
we assume that these spaces can approximate smooth functions as described in \eqref{w_ell_7_TD}, \eqref{w_ell_7_hSm_TD}, 
and Remark \ref{Rmk:NNapproximation}.

Section \ref{Se:4} is dedicated to deriving maximal regularity estimates within an abstract framework for evolution equations 
 of parabolic type in a Hilbert space. The results presented in this section are of independent interest, both for their proof 
 techniques and because they hold  for variable time steps, unlike the results in \cite{kovacs2016stable}.

We present all time-discrete methods in a unified point-wise formulation as given in \eqref{eq:nm-abstr1}. In Section \ref{SSe:Gal}, we focus 
on Galerkin time-stepping methods, including continuous and discontinuous Galerkin methods. Section \ref{SSe:1.3} addresses collocation 
Runge–Kutta methods, with detailed proofs provided for Gauss and Radau IIA methods, as well as for all algebraically stable Runge–Kutta 
methods. Notably,   we derive a novel connection of Lobatto IIIA to continuous Galerkin methods and thus we are able to establish its maximal 
regularity bounds. 

It is worth highlighting that Lobatto IIIA methods are high-order extensions of the trapezoidal rule. Despite being A-stable, their coefficient matrix 
is not invertible, and they lack 
B-stability. As a result, our maximal regularity estimates appear to be the first in the literature 
for this significant class of Runge–Kutta methods.

Section \ref{Se:5} is dedicated to numerical experiments. In addition to parabolic equations, we also consider the wave equation. While the convergence 
theory for hyperbolic equations remains to be developed, the formulation of the methods can be directly applied by expressing the equation as a first-order 
evolution system in time.

The numerical results indicate that our proposed methods exhibit the desired properties. They achieve higher accuracy and, more importantly, 
allow for the adaptation of their qualitative characteristics to produce physically relevant approximations. We did not include computations based 
on explicit Runge–Kutta discretizations. Instead, we refer to \cite{GGM_pinn_2023}, where neural network methods utilizing the \emph{explicit Euler} 
scheme fail to meet the stability criteria [S1] and [S2], resulting in unstable behavior. To regain stability, these methods require the enforcement 
of standard CFL conditions that link the space and time discretization parameters.

For training in the spatial variable, we employed quasi-Monte Carlo sampling. This was done both for comparison with full (space and time) 
quasi-Monte Carlo sampling and to underscore that our approach is well-suited for high-dimensional evolution PDEs. Notably, the time discretization 
introduces a mesh only in the time variable, making the method particularly effective in high dimensions.
 
%

\subsubsection*{\it Remarks on bibliography}
Physics Informed Neural Networks is a class  of neural network based methods to approximate solutions of PDEs, \cite{Karniadakis:pinn:orig:2019}.  
The loss is based on the  residual of the PDE; similar    methods were considered in   \cite{Lagaris_1998}, \cite{Berg_2018},  \cite{Raissi_2018},  
\cite{SSpiliopoulos:2018}, \cite{makr_pryer2024deepUz}.    Different loss functionals leading to various neural network methods for differential 
equations  are considered in \cite{kevr_1992discrete}, \cite{e2017deep}, \cite{kharazmi2019variational}, \cite{Xu}, 
\cite{chen2022deep},  \cite{B_Canuto_P_Vpinn_quadrature_2022}
\cite{georgoulis2023discrete}, 
\cite{Grohs:space_time:2023}.    Previous works on the analysis of these methods include \cite{SSpiliopoulos:2018}, \cite{Mishra_dispersive:2021}, 
\cite{Karniadakis:pinn:conv:2019}, \cite{shin2020error}, \cite{Mishra:pinn:inv:2022,Mishra_gen_err_pinn:2023}, \cite{makr_pryer2024deepUz}, \cite{makr_pryer2024uzLagr}.

Neural network approximations for evolution equations have been previously studied in \cite{SSpiliopoulos:2018}, \cite{Karniadakis:pinn:orig:2019}, 
\cite{Mishra_dispersive:2021}, \cite{Biswas:2022aa}, \cite{georgoulis2023discrete}, and \cite{Grohs:space_time:2023}. In the  works \cite{SSpiliopoulos:2018}, \cite{Karniadakis:pinn:orig:2019}, \cite{Mishra_dispersive:2021},  a global space-time residual-type loss was employed, whereas in \cite{georgoulis2023discrete} 
and \cite{Grohs:space_time:2023}, a time-stepping approach based on the backward Euler or Minimizing Movement schemes was adopted. 
Discrete time models related to Runge--Kutta methods were considered in 
\cite{Karniadakis:pinn:orig:2019}.
%
%
%
%

As mentioned, our stability and convergence analysis follows the framework introduced in \cite{GGM_pinn_2023}, which 
is motivated by $\varGamma$-convergence arguments. In \cite{muller2020deep}, 
$\varGamma$-convergence was employed to analyze deep Ritz methods without considering training. More recently, the 
$\liminf - \limsup \,$ framework was utilized in \cite{loulakis2023newap} to establish convergence results for global and local 
discrete minimizers in general machine learning algorithms with probabilistic training. Additionally, this framework was applied 
in \cite{grekas2024deepritzfinite} to analyze deep Ritz methods trained with finite element techniques.
   


\noindent

\section{Motivation and problem formulation}\label{Se:2}

\subsection{Model problems and their  Machine Learning approximations}\label{Se:1NN}
We will  consider linear evolution   PDEs. The formulation of the methods can be applied to 
parabolic or wave type time dependent equations, linear or nonlinear,  and our main focus is on the effect of the time discretization mechanism. 

\subsection{A linear evolution PDE}
To motivate our approach we consider as a model problem a linear  parabolic equation. 
We use the compact notation $ \varOmega_T = \varOmega \times (0,T] ,$
 for some fixed time $ T>0.$
We consider the initial and boundary value problem
\begin{equation}\label{ParabolicPDE-I}
\left\{
\begin{alignedat}{3}
&u_t + Lu = f  \qquad&&\text{in }  &&\varOmega_T, \\
& u =  0 \qquad&&\text{on }  &&\partial \varOmega \times (0,T] , \\
&  u =  u_0\qquad&&\textrm{in }  &&\varOmega ,
\end{alignedat}
\right.
\end{equation}
%
where $ f \in L^2( \varOmega_T) ,\; u_0 \in H^1_0(\varOmega) $, and $ L $ is a coercive,
self-adjoint, second order elliptic operator. 
The associated energies used will be the $L^2$-residuals  
\begin{equation}\label{Functional}
\mathcal{E}(v) = \int_0^T \int_{\varOmega} |   v_t + Lv  - f |^2 \d  x\, \d t +  |  v(0)  - u_0 |_ {H^1} ^2 
\end{equation}
defined over smooth enough functions and  domains  $\varOmega .$ The choice of the $H^1$ seminorm for the initial  
condition,  $  |  v(0)  - u_0 |_ {H^1},$ is done in order to obtain a balanced energy which is particularly convenient 
in the analysis; obviously, other choices are possible.

\subsubsection*{Nonlinear Spaces generated by Neural Networks}
Physics Informed Neural Networks are based on the minimization of the functional $\mathcal{E}$ 
over a chosen discrete set consisting of neural networks, aiming at approximating $u. $ To fix ideas, we consider functions 
$u_\theta$ defined through neural networks. 
The construction below is typical, and it is presented for completeness. Our results only depend   on the  approximation capacity of these functions. 
A \emph{deep neural network} maps every point $\overline x\in \varOmega \times [0, T] $ to  $u_\theta (\x) \in \R$, through
\begin{equation}\label{C_L}
	u_\theta(\x)= C_L  \circ \sigma  \circ C_{L-1} \cdots \circ\sigma \circ C_{1} (\x) \quad \forall \x\in \varOmega _T.
\end{equation}
%
Any such map $\mathcal{C}_L$ is defined  by the intermediate (hidden) layers $C_k$, which are affine maps of the form 
\begin{equation}\label{C_k}
	C_k y = W_k y +b_k, \qquad \text{where }  W_k \in \R ^ {d_{k+1}\times d_k}, b_k \in \R ^ {d_{k+1}},
\end{equation} 
where the dimensions $d_k$ may vary with each layer $k$ and $\sigma (y)$ denotes the vector with the same number of components as $y$,
where $\sigma (y)_i= \sigma(y_i)\, .$ 
The index $\theta$ represents collectively all the parameters of the network $\mathcal{C}_L,$ namely $W_k , b_k, $ $k=1, \dots, L .$ 
The set of all networks $\mathcal{C}_L$ with a given structure (fixed $L, d_k,  k=1, \dotsc, L\,$) of the form \eqref{C_L}, \eqref{C_k}
is called $\mathcal{N}.$ The total dimension (total number of degrees of freedom) of  $\mathcal {N} ,$ is $\dim{\mathcal {N}}= \sum _{k=1} ^L d_{k+1} (d_k +1)  .$ 
We now define the nonlinear discrete set  of functions 
\begin{equation}
	V _{\mathcal{N}}= \{ u_\theta : \varOmega _T \to \R ,  \ \text{where }  u_\theta (\x) = \mathcal{C}_L (\x), \ \text{for some } \mathcal{C}_L\in  \mathcal{N}\, \}  . 
\end{equation}
Boundary conditions is a subtle issue. To avoid extra technical problems,   it will be useful to introduce, following  \cite{sukumar2022exact}, the set of functions  
which exactly satisfy the boundary conditions through appropriate distance functions depending only on the domain 
$\varOmega  . $ If $\varPhi$ is such a function, see \cite[Section 5.1.1]{sukumar2022exact}, we define 
\begin{equation}\label {V_N0}
	V _{\mathcal{N}, 0}= \{ u_\theta : \varOmega_T  \to \R ,  \ \text{where }  u_\theta (x, t) =  \varPhi (x)\mathcal{C}_L (x, t), \ \text{for some } \mathcal{C}_L\in  \mathcal{N} \,  \} .
\end{equation}
Then, $V _{\mathcal{N}, 0} \subset H^1((0,T) ; L^2 (\varOmega)) \cap L^2((0,T) ; H^2(\varOmega) \cap H^1_0(\varOmega)),$ for smooth enough activation function $\sigma.$

%

\subsubsection*{Abstract Loss -- minimization  on  $V _{\mathcal{N}}$} Physics Informed Neural networks 
are based on the minimization of residual-type functionals of the form \eqref{Functional} over discrete 
neural network sets of given architecture. To this end, we assume that the (non-computable) abstract  problem 
\begin{equation}\label{mm_nn:abstract}
	\min  _ {v \in  V _{\mathcal{N}, 0} } \E (v)
\end{equation}
possesses a solution $v^\star \in V _{\mathcal{N}, 0}  .$ 
%
%
The integrals appearing in the loss functional $\E $ require further discretization to result in a computable loss.  
As the set $ V _{\mathcal{N}, 0} $  is nonlinear, the problem needs to be considered as minimization in the parameter space over  $\R ^{\dim{\mathcal {N}}}$ 
%
\begin{equation}\label{mm_nn:abstract_theta}
	\min _ { \theta \in  \Theta } \E(u_\theta), 
\end{equation}
which turns out to be  non-convex with respect to $\theta$ even though the functional $\E (v)$ is convex with respect to $v.$ 

\subsection{Simple time-discrete Training} To formulate fully discrete   schemes, we shall need  computable discrete 
versions of the energy $\E(u_\theta).$  This can be achieved through deterministic or probabilistic training.

 \subsubsection*{Deterministic and probabilistic training} We consider appropriate quadrature for integrals over 
 $\varOmega_T$ ({Training through quadrature}). Such a quadrature requires 
 a set $K_h$ of discrete points $z\in K_h$  and corresponding nonnegative weights $w_z$
 such that 
 \begin{equation}
\label{quadrature}
\sum _{z\in K_h} \, w_z \, g(z) \approx \int _{\varOmega_T} \, g(\x) \, \d \x .
\end{equation}
  Then, one can define the discrete functional 
\begin{equation}
\label{E_h}
\mathcal{E}_{Q, h}( g )  = \sum _{z\in K_h} \, w_z \,  
| v_t (z) + L v (z) - f(z) |^2\,  \, . 
\end{equation}
The initial condition is discretized in a similar way.  

An alternative is to approximate integrals 
using  probabilistic (Monte Carlo, Quasi-Monte Carlo)  quadrature rules. To this end, 
we may  consider a collection $ X_1,X_2,\ldots$ of i.i.d.\ $\varOmega _T$-valued random variables, defined on an appropriate probability space 
corresponding to sample points in $\varOmega_T\, .$
Let  $\omega $   be a fixed instance, and $X_i(\omega )\in \varOmega_T$ be the corresponding values of the random variables. 
Monte Carlo approximation of the space-time integral yields the discrete sum 
\begin{equation}\label{prob_E}
\mathcal{E}_{N, \omega }( v )  =  \frac 1N  \sum _{i=1 } ^N    \,  | v_t (X_i(\omega) ) + L v (X_i(\omega) ) - f(X_i(\omega) )| ^2 .     
 \end{equation}
 The discrete minimization problem for each instance is 
\begin{equation}\label{ieE-minimize_NNQ}
	\min  _ {v \in  V _{\mathcal{N}, 0 }} \E _{N, \omega }(v)\,  +  \E ^0 _{N, \omega }(v) ,  
\end{equation}
where $\E ^0 _{N, \omega }(v)$ is a Monte Carlo approximation of the initial condition. One of the main advantages 
of these discretizations is that they scale reasonably with the dimension, and are thus preferable for high-dimensional operators $L .$

  \subsubsection*{Time discrete training} To introduce the Runge--Kutta PINN algorithms, it will be  instrumental  
  to consider a hybrid approach where quadrature (and discretization) is applied only to the time variable of the
   time dependent problem.  Then, the fully discrete scheme can be designed using alternative discretizations in space, 
   deterministic or probabilistic.

 To fix notation, let $0=t_0< t_1< \cdots <t_N=T$ define a partition of $[0,T]$ and $J_n:=(t_{n},t_{n+1}],$ $k_n:=t_{n+1}-t_{n}.$
We   denote by $v_m(\cdot)$ and $f_m(\cdot)$ the values $v (\cdot,t_m)$ and $f(\cdot ,t_m).$ Then, we consider 
the discrete in time quadrature 
 \begin{equation}
\label{quadrature_time}
\sum _{n=0}^{N-1} \,  k_n \, g(t_{n+1}) \approx \int_0^T  g(t) \, \d t.
\end{equation}
We proceed to define the time-discrete version of the functional \eqref{Functional} as follows
\begin{equation}\label{Functional_k}
\mathcal{G}_{k, IE} (v) = \sum _{n=0}^{N-1} \,  k_n \, \int_{\varOmega }\big | \frac{v_{n+1}-  v_{n}}{k_n} + L v_{n+1}  - f(t_{n+1}) \big |^2 \, \, \d x 
+    |  v_0  - u_0 |_{H^1(\varOmega)} ^2 \, . 
\end{equation}
In \cite{GGM_pinn_2023}, it was shown that the problem
\begin{equation}\label{ieE-minimize_k}
	\min_ {v \in  V _{\mathcal{N}} } \mathcal{G} _{k, IE}(v)  
\end{equation}
%
%
yields stable and convergent approximations to the exact solution as opposed to the  analogue (from the point of view of quadrature 
and approximation) discrete functional:
\begin{equation}\label{Functional_k_Ex}
\mathcal{G}_{k, EE} (v)  \sum _{n=0}^{N-1} \,  k_n \, \int_{\Omega }\big | \frac{v_{n+1}-  v_{n}}{k_n} + L v_{n}  - f(t_n) \big |^2 \, \, \d x +    |  v_0  - u_0 |_{H^1(\varOmega)} ^2\,.  \end{equation}
Thereby hinting that there is a deeper connection between standard stability notions of time discretizations and this class of neural network algorithms. 
In the analysis, it was instrumental to consider an alternative representation of the time discrete functional through reconstructions. 
This can motivate   the design of the Runge--Kutta PINN methods of the next section. 
Indeed,  let the \emph{linear time reconstruction}
   $\widehat v$ 
 of a time dependent function  $v$  be the piecewise linear approximation   
of $v$ defined by linearly 
interpolating between the nodal values $v_{n}$ and $v_{n+1}$: 
\begin{equation}\label{def_hatU}
\widehat v(t):=\ell_0^n(t) v_{n}+ \ell_1^n(t) v_{n+1}, \quad t\in J_n,
\end{equation}
with
$\ell_0^n(t) := (t_{n+1}-t)/k_n$  and  $\ell_1^n(t) := (t-t_{n})/k_n.$ 
If  the piecewise constant interpolant of $v_{j}$  is denoted by $\overline v, $  
\begin{equation}\label{def_overU}
\overline v(t):= v_{n+1}, \quad t\in J_n\, ,\end{equation}
the time discrete energy  $ \mathcal{G}_{IE, k}   $ becomes
\begin{equation}\label{ParabFunctional_k}
\begin{split}
	\mathcal{G}_{IE, k} (v) &= \| \widehat v_t + L \overline v - \overline f \|^2_{L^2((0,T); L^2(\varOmega))}  + |   \widehat v(0)  - u_0 |_{H^1(\varOmega)}^2 \\
&= \int_0^T \|  \widehat v_t + L \overline v - \overline f \|^2_{L^2(\varOmega)}\, \d t  +   |\widehat v(0)  - u_0 |_{H^1(\varOmega)}^2  .
\end{split}
\end{equation}
This representation of the loss will be generalized to high-order time discretizations in the next section. 

\subsection{Runge--Kutta Physics Informed Neural Networks} 
To introduce the Runge--Kutta PINNs, we first recall the connection of Runge--Kutta methods and collocation time discretizations.

\subsubsection{Collocation Runge--Kutta methods}\label{SSe:1.3}
For $q\in \N,$ let $0\leqslant c_1<\dotsb<c_q\leqslant 1$ denote the intermediate nodes of a Runge--Kutta method or collocation nodes.
With starting value $U_0=u_0,$ we    consider the discretization of   problem \eqref{ParabolicPDE-I}   
by a $q$-stage Runge--Kutta method: we recursively define approximations 
$U_\ell\in V$ to the nodal values $u(t_\ell)$, as well as internal approximations $U_{\ell i}\in V$ 
to the intermediate  values $u(t_{\ell i}), t_{\ell i}=t_\ell+c_ik_\ell,$ by
\begin{equation}\label{eq:RK1-I}
\left \{
\begin{alignedat}{2}
&U_{ni}=U_n- k _n \sum_{j=1}^q a_{ij} \big (L U_{nj} -f(t_{nj})\big ),\quad &&i=1,\dotsc,q,\\
&U_{n+1} =U_n- k _n \sum_{i=1}^q b_i \big (L U_{ni} -f(t_{n i})\big ), \quad &&
\end{alignedat}
\right . 
\end{equation}
$n=0,\dotsc,N-1.$ Here, without being very precise, $V$ is a functional space where our approximations are sought for every $t\in [0, T];$ 
typically $V$ is included in the domain of the operator $L.$   
The most important class of Runge--Kutta methods are the \emph{collocation Runge--Kutta} methods, \cite{GuillouSoule_1969}, 
\cite{Wright_1970}, \cite{HairerNW1993}. Such methods are equivalent to considering  a collocation approximation $\widehat U $ 
which is a continuous piecewise polynomial function of local degree $q$   satisfying    
$\widehat U(0)=u_0$ and  the collocation conditions
\begin{equation}
\label{eq:coll-1-I}
\widehat U'(t_{ni})+L\widehat U(t_{ni})=f(t_{ni}), \quad i=1,\dotsc,q, \quad n=0,\dotsc,N-1.
\end{equation}
The $q$-stage Runge--Kutta method and the collocation points are related through the relations
\begin{equation}
\label{RK-coef}
a_{ij} =\int _0^{c_i} \ell_j (\tau )\, \d \tau, 
\quad b_{i} =\int _0^{1} \ell_i(\tau )\, \d \tau,
\quad   i,j=1,\dotsc,q;
\end{equation}
here, $\ell_1,\dotsc, \ell_q\in \P_{q-1}$ are the Lagrange polynomials for the
collocation nodes $c_1,\dotsc,c_q,$ $\ell_i(c_j)$ $=\delta_{ij}, i,j=1,\dotsc,q.$
In this case, the \emph{stage order} of the Runge--Kutta method is $q.$ 

It is well known,  \cite{GuillouSoule_1969}, \cite{Wright_1970}, \cite{HairerNW1993}, 
 that  the collocation and Runge--Kutta methods \eqref{eq:coll-1-I} and \eqref{eq:RK1-I},
respectively, are equivalent in the sense that they yield the same approximations at the nodes
and at the intermediate nodes, i.e.,
\begin{equation}\label{eq:equivalent-I}
\begin{alignedat}{2}
&\widehat U(t_n)=U_n, \quad &&n=1,\dotsc,N,\\
&\widehat U(t_{ni})=U_{ni}, \quad &&i=1,\dotsc,q, \ \, n=0,\dotsc,N-1.
\end{alignedat}
\end{equation}
%
%
%
%
%
%
%
%
These methods admit a crucial pointwise formulation which is the key element of the loss functional defined below. Indeed, 
as in    \cite{AMN2} and \cite{AMN3}, if we let $I_{q-1}$ be the piecewise interpolation operator by polynomials of degree $q-1$ 
at the collocation nodes $t_{ni}, i=1,\dotsc,q, n=0,\dotsc,N-1,$
and because  $\widehat U'$ is a piecewise polynomial of degree $q-1$ as well, 
we can write \eqref{eq:coll-1-I} in \emph{pointwise form} as 
\begin{equation}
\label{eq:coll-2-I}
\widehat U'(t)+I_{q-1} L\widehat U(t)=I_{q-1} f(t), \quad t\in (t_n,t_{n+1}], \quad n=0,\dotsc,N-1.
\end{equation}
The interpolants $U:=I_{q-1} \widehat U$ and $I_{q-1} f$ are piecewise  
polynomials of degree $q-1$ and, in general, discontinuous at the nodes $t_0,\dotsc,t_{N-1}.$
The pointwise form \eqref{eq:coll-2-I} of the numerical method will be the basis for defining the Runge--Kutta discrete loss.

%
%
%

\subsubsection{Runge--Kutta discrete Loss}
\label{SSe:2.4.2} 
We shall introduce more notation related to the 
representation of piecewise polynomial functions: 
Let  $\ell_{n1},\dotsc, \ell_{nq}\in \P_{q-1}$ be the Lagrange polynomials $\ell_1,\dotsc,  \ell_q\in \P_{q-1}$
for the collocation nodes $c_1,\dotsc,c_q$ 
shifted to  the interval $[t_n,t_{n+1}], \ell_{ni}(t)=\ell_{ni}(t_n+k_n\tau) = \ell_i(\tau),  i=1,\dotsc,q.$ 
Obviously, 
\[ \ell_{ni}(t_{nj})=\delta_{ij},\quad i,j=1,\dotsc,q.\]  
Furthermore, we let $0=\tilde c_0<\dotsb< \tilde c_q=1$ be auxiliary points, such that   $\tilde c_0 =0$ and $\tilde c_q=1,$  and  let 
${\tilde \ell}_{n0},{\tilde \ell}_{n1},\dotsc, {\tilde  \ell}_{nq}\in \P_q$ be the Lagrange polynomials 
\[{\tilde \ell}_0,{\tilde \ell}_1\dotsc,  {\tilde \ell}_q\in \P_q\]
 for the points $\tilde c_0,\dotsc,\tilde c_q$ 
shifted to  the interval $[t_n,t_{n+1}], {\tilde \ell}_{ni}(t)={\tilde \ell}_{ni}(t_n+k_n\tau) = {\tilde\ell}_i(\tau),  i=0,\dotsc,q.$ 
Obviously, ${\tilde \ell}_{ni}(\tilde t_{nj})=\delta_{ij}, i,j=0,1,\dotsc,q,$ with $\tilde t_{nj}:=t_n+k_n\tilde c_j,  j=0,\dotsc,q.$ 
With this notation, let the \emph{interpolant}
   $\widehat v$ 
 of a time dependent function  $v$  be the piecewise polynomial function approximating  $v$ defined by  
interpolating   the nodal values $v(\tilde t_{nj})$ of $v$ as follows
\begin{equation}
\label{eq:Lagr6}
\widehat v (t)=\widehat I_{q} v (t)= \sum_{i=0}^q {\tilde \ell}_{ni}(t) v(\tilde t_{ni}),\quad 
\overline v (t) = I_{q-1}\widehat v (t)=\sum_{j=1}^q \ell_{nj}(t) \widehat v(t_{nj}),
\quad t\in (t_n,t_{n+1}].
\end{equation}
%
Notice that in principle $ \tilde t_{nj} $ could be different from $ t_{nj},$ and we can choose them at our convenience.
 The reason of introducing $ \tilde t_{nj}$ is that we would like to interpolate  
neural network functions $v $ on the space of \emph{continuous} piecewise polynomial functions and 
apply afterwards the collocation residual to $\widehat v.$
In some important cases $ \tilde t_{nj} $ can be chosen as  extensions of the collocation points $t_{nj} $ 
by including an additional node; see Remark \ref{assumptions_Pi}.



We are ready to define the  Runge--Kutta discrete loss by 
\begin{equation}\label{RK_ParabFunctional}
\begin{split}
	\mathcal{G}_{RK} (v) &= \| \widehat v_t(t)+I_{q-1} L\widehat v(t)-I_{q-1} f(t) \|^2_{L^2((0,T); L^2(\varOmega))}  
	+ |    \widehat v(0)  - u_0 |_{H^1(\varOmega)}^2 \\
&= \int_0^T \| \widehat v_t (t)+I_{q-1} L\widehat v(t)-I_{q-1} f(t)\|^2_{L^2(\varOmega)}\, \d t  + |   \widehat v(0)  - u_0 |_{H^1(\varOmega)}^2 .
\end{split}
\end{equation}
The Runge--Kutta Physics Informed Neural Network method is based on the discrete minimization problem for the loss $\mathcal{G}_{RK} (v), $
\begin{equation}\label{ieE-minimize_NNQ}
	\min  _ {v \in  V _{\mathcal{N}, 0 }} \mathcal{G}_{RK} (v) . 
\end{equation}
The formulation of the method for nonlinear evolution problems is straightforward by using similar interpolation operators; see \cite{AMN2}.

\subsubsection{A general discrete Loss}
\label{SSe:2.4.3} 
The Runge--Kutta time discretizations as well as
continuous and discontinuous time-Galerkin sche\-mes can be cast into a unified formulation showing 
that the proposed time discrete training can be quite general. This formulation  shall be used in the analysis of Sections \ref{Se:3} and \ref{Se:4}. 
We consider   generalized interpolation operators as follows: 
For $q\in \N,$ we consider  projection or interpolation operators $\varPi_{q-1} ,  {\widetilde \varPi}_{q-1} $  to the piecewise 
polynomial functions of local degree $q-1.$ 
Furthermore, we consider a projection or interpolation operator 
\[\widehat v = {\widehat \varPi}_{q}v  
\]
 of a time dependent function  $v$  to be the piecewise polynomial function of local degree $q$ approximating   
 $v.$  We assume that if $v$ is continuous in time, then $\widehat v  $ 
is a piecewise polynomial globally continuous in time function. The standard interpolation operators as defined 
in Section \ref{SSe:2.4.2} are typical cases of these operators. 
%
%
%
We are ready to define the  time-discrete  loss by 
\begin{equation}\label{RK_ParabFunctional}
	\mathcal{G}_{k} (v) = \int_0^T  \| \widehat v_t(t)+\varPi_{q-1}  L\widehat v(t)-{\widetilde \varPi}_{q-1}f(t) \|^2_{L^2(\varOmega)}\, \d t  
	+ |   \widehat v(0)  - u_0 |_{H^1(\varOmega)}^2 .\\
\end{equation}
In Section    \ref{Se:4} we show systematically that all Runge--Kutta methods considered and both continuous and discontinuous 
Galerkin time discrete schemes are associated to the 
general formulation involving the discrete operators $\varPi_{q-1} ,  {\widetilde \varPi}_{q-1} $ and $ {\widehat \varPi}_{q}.$

The generalized Runge--Kutta Physics Informed Neural Network method is based on the discrete minimization problem for the loss $\mathcal{G}_{k} (v), $
\begin{equation}\label{ieE-minimize_general}
	\min  _ {v \in  V _{\mathcal{N}, 0 }} \mathcal{G}_{k} (v) . 
\end{equation}
Under certain assumptions on the generalized operators (see Proposition \ref{TD:EquicoercivityofG}, \eqref{exact_integrals}, and \eqref{eq:Lagr6}), 
we conduct the stability and convergence analysis for \eqref{ieE-minimize_general} in Section \ref{Se:3}. These assumptions are shown 
to be satisfied by all the methods discussed in this work, as demonstrated in Section \ref{Se:4}.

\section{Stability and Convergence for Parabolic equations}\label{Se:3}

Let as before $ \varOmega \subset \mathbb{R}^d $ be open and bounded, and set $ \varOmega_T = \varOmega \times (0,T] $ for some fixed time $ T>0. $
We consider the parabolic problem
\begin{equation}\label{ParabolicPDE}
\left\{
\begin{alignedat}{3}
&u_t + Lu = f  \qquad&&\text{in }  &&\varOmega_T, \\
& u =  0 \qquad&&\text{on }  &&\partial \varOmega \times (0,T] , \\
&  u =  u_0\qquad&&\textrm{in }  &&\varOmega \times \lbrace t=0 \rbrace .
\end{alignedat}
\right.
\end{equation}
 In  this section we discuss  convergence properties of approximations of \eqref{ParabolicPDE} obtained by minimization 
of continuous and time-discrete energy functionals over appropriate sets of  neural network functions. We shall assume 
that $\varOmega$ is a convex Lipschitz domain. 
This assumption is made to ensure that elliptic regularity estimates are valid. The case of a non-convex domain can 
be treated with the appropriate modifications in the analysis. 

The continuous functional can be defined as follows: Consider  
\[ \mathcal{G} : H^1((0,T) ; L^2 (\varOmega)) \cap L^2((0,T) ; H^2(\varOmega) \cap H^1_0(\varOmega))\rightarrow \overline{\mathbb{R}} \]
 such that 
\begin{equation}\label{ParabFunctional}
\mathcal{G}(v) =
\int_0^T \| v_t(t) + L v(t) - f(t) \|^2_{L^2(\varOmega)} \d t +  | v(0) - u_0  |^2_{H^1(\varOmega)}\, .
\end{equation}
The use of the ${H^1(\varOmega)}$ seminorm for the initial condition is more appropriate for stability purposes for  parabolic equations. 
While weaker choices are certainly possible, they would require a modified technical analysis.

\subsection{Time discrete training through Runge--Kutta and time Galerkin}

%
%
We shall work with the general    time-discrete 
loss defined in Section \ref{SSe:2.4.3}: 
\begin{equation*}\label{RK_ParabFunctional}
	\mathcal{G}_{k} (v) = \int_0^T  \| \widehat v_t(t)+\varPi_{q-1}  L\widehat v(t)-{\widetilde \varPi}_{q-1}f(t) \|^2_{L^2(\varOmega)}\, \d t  
	+ |   \widehat v(0)  - u_0 |_{H^1(\varOmega)}^2 .\\
\end{equation*}
In the sequel, we shall use the compact notation 
\begin{equation}\label{def_overU}
\widehat U = {\widehat \varPi}_{q}U,\quad    \overline U:= \varPi_{q-1} \widehat U, \quad \text{and} \quad \widetilde  f = {\widetilde \varPi}_{q-1}f.
\end{equation}

The neural network spaces are selected to meet specific approximability criteria aligned with established results in approximation theory; 
see, e.g.,  \cite{Grohs_space_time_approx:2022, Mishra:appr:rough:2022, Xu} and their references.  However, subtle challenges arise, 
particularly with boundary conditions and the fact that existing approximation results 
do not yet offer concrete guidance on selecting specific architectures. Nevertheless, to investigate the potential convergence of minimizers, 
we assume that the following approximability requirements are satisfied.
The required smoothness of the spaces is guaranteed by selecting smooth enough activation functions. The neural network spaces are 
selected such that for each 
  $\ell \in \mathbb N$ we associate  a  space  $ V _{\mathcal{N}, 0}  ,$
  which is denoted by  $V_\ell $ with the approximation property: For each $w\in H^1((0,T); L^2(\varOmega))\cap L^2((0,T); H^2(\varOmega) \cap H^1_0(\varOmega)) $ 
  there exists a $w_\ell \in V_\ell$ such that
    \begin{equation}\label{w_ell_7_TD}
 \|w_{\ell}-w\|_{H^1((0,T); L^2(\varOmega))\cap L^2((0,T); H^2(\varOmega))}  \leqslant  \beta _\ell \, (w),
  \qquad 
 \text{and } \ \beta _\ell \, (w) \to 0,  \  \ \ell\to \infty\, .	
\end{equation}
If in addition, $w $ has higher regularity, we assume that   
\begin{equation}\label{w_ell_7_hSm_TD}
 \|(w_{\ell}-w)' \|_{H^1((0,T); L^2(\varOmega))\cap L^2((0,T); H^2(\varOmega))}   \leqslant  \tilde \beta _\ell \, \| w'\| _{ H^m((0,T); H^2(\varOmega))}  , \qquad 
 \text{and } \ \tilde \beta _\ell \, \to 0,  \  \ \ell\to \infty\, ,	
 \end{equation}
 where in the above relation and throughout this section, the time derivative is denoted by  $w',$ i.e., $w' :=w_t\, .$ 

\begin{remark}\label{Rmk:NNapproximation}
 The current state of the art in approximating smooth functions using neural network spaces lacks sufficient information regarding the specific 
 architectures necessary to achieve particular bounds and rates. The above  assumptions can be relaxed by requiring that \eqref{w_ell_7_TD} 
 and \eqref{w_ell_7_hSm_TD} hold only for $w=u,$ where $u$ represents the exact solution of the problem.
\end{remark}


With the spaces $V_\ell$ defined above, we shall use the following notation for the discrete energies:
\begin{equation}\label{GdeltaEnergies}
\mathcal{G}_{\ell}(v_{\ell}) = 
\left\{
\begin{alignedat}{2} 
&  \mathcal{G}_{k(\ell)} (v_\ell) ,\quad && v_\ell \in V_\ell, 
\\ 
&+ \infty ,\quad && \textrm{otherwise.}
\end{alignedat}
\right.
\end{equation}
Here $k=k(\ell)$ are selected just to satisfy  $k=k(\ell)\to 0 $
as $\ell\to \infty\, , $ and $\mathcal{G}_{k(\ell)} (v_\ell) =\mathcal{G}_{k} (v_\ell)$ is defined by \eqref{RK_ParabFunctional}.
%
%

\subsubsection{Stability-Maximal Regularity}\label{Se:3.1.1}

Following \cite{GGM_pinn_2023}, we call our methods \emph{stable} if 
two key properties,  roughly stated as follows, hold 
\begin{enumerate}
	\item [{[S1]}]  If the energies  $\mathcal{E}_\ell$ are uniformly bounded
\[\mathcal{E}_\ell [u_\ell] \leqslant C,\]  
then  there exists a constant $C_1>0$ and $\ell$-dependent norms $V_\ell$ such that 
\begin{equation}
\|u_\ell\|_{V_\ell} \leqslant C_1 . \label{coer:dg_seminorm}
\end{equation}
\item [{[S2]}]  Uniformly bounded sequences in $\|\cdot \| _{V_\ell}$ have convergent subsequences in $H.$
\end{enumerate}
Here, 
$H$ is a normed space (typically a Sobolev space) that depends on the form of the discrete energy being 
considered. Property [S1] requires that 
$ \mathcal{E}_\ell [v_\ell] $  is coercive with respect to (potentially 
$\ell$-dependent) norms or semi-norms. Moreover, [S2] implies that, although the norms 
$\|\cdot \| _{V_\ell}$ are 
$\ell$-dependent, they should allow the extraction of convergent subsequences from uniformly bounded sequences 
in these norms, in a weaker topology induced by the space $H.$

This definition is inspired by a discrete interpretation of the Equi-Coercivity property in the 
$\varGamma$-convergence  framework in the calculus of variations. As we will demonstrate, this property is fundamental 
to establishing compactness and the convergence of minimizers for the approximate functionals, as shown later in this section.

The stability of $ \mathcal{G}_{k} $ follows by the next result which hinges on the maximal regularity estimates established in Section \ref{Se:4}.

\begin{proposition}\label{TD:EquicoercivityofG} Assume that the following maximal regularity estimate 
is satisfied
\begin{equation}
\label{eq:MR_Se:3}
\|\overline U\|_{L^2((0,T); H^2(\varOmega))} + \|\widehat U'\|_{L^2((0,T); L^2(\varOmega))} \leqslant C\Big [ \| \widehat U(0)\|_{H^1(\varOmega)}
+\| \widehat U_t + L \overline U \|_{L^2((0,T); L^2(\varOmega))}\Big ];
\end{equation} 
then, the functional $ \mathcal{G}_{k}$ defined in \eqref{RK_ParabFunctional} is stable with respect to $ \widehat U , \overline U $ in the following sense: 
\begin{equation}\label{GEquicoercivity_TD}
\begin{gathered}
\textrm{If} \;\; \mathcal{G}_k(U) \leqslant C \;\; \textrm{for some} \;\: C >0 \;,\; \textrm{we have} \\
\|\overline U\|_{L^2((0,T); H^2(\varOmega))} + \|\widehat U'\|_{L^2((0,T); L^2(\varOmega))} \leqslant C_{MR}.
\end{gathered}
\end{equation}
\end{proposition}
$ \\ $

\begin{proof}
Since 
\begin{equation}\label{Proofof G EquicoercEq1_1} 
\int_0^T  \| \widehat U_t + L \overline U -{\widetilde \varPi}_{q-1}f(t) \|^2_{L^2(\varOmega)}\, \d t\,\leqslant C, 
\end{equation}
 we have 
\begin{equation}\label{Proofof G EquicoercEq4}
\| \widehat U_t + L \overline U \|_{L^2((0,T); L^2(\varOmega))} \leqslant C_1 . 
\end{equation}
Therefore, in view of \eqref{eq:MR_Se:3} and the fact that 
\[\| \widehat U(0)\|_{H^1(\varOmega)} \leqslant C \]
as a result of $\mathcal{G}_k(U) \leqslant C,$
we conclude the proof.
\end{proof}

\subsection{Convergence of the minimizers}
 Next, we shall prove that    the sequence of discrete minimizers $ (u_\ell) $ converges in $L^2((0,T); H^1(\varOmega))$ to
  the    minimizer of the continuous problem.  
 We first show a $\liminf$  and a $\limsup$ inequality. 
 
 \begin{lemma}[{$\liminf$ inequality}]\label{TD:liminf}
Assume that a sequence  $\{U_\ell \},$ $ U_\ell \in V_\ell $, satisfies 
\[ \mathcal{G}_{ \ell} (U_\ell) \leqslant C \]
uniformly in $ \ell .$
Let  the operators $\varPi_{q-1} ,  {\widehat \varPi}_{q} $ 
satisfy
\begin{equation}\label{exact_integrals}
\int _{J_n} 	{\widehat \varPi}_{q}v \, \d t  = \int _{J_n} \varPi_{q-1} {\widehat \varPi}_{q} v \, \d t \, , 
\end{equation} 
for all $v \in V _\ell. $
Assume further that there exists 
a $\tilde u\in H^1((0,T); L^2(\varOmega))\cap L^2((0,T); H^2(\varOmega))$ such that 
\[ U_\ell \to \tilde u, \quad \ell \to \infty, \quad \text { in } L^2((0,T); H^1(\varOmega))\, ;\]
then, 
\begin{equation}\label{Glowersemicontinuity}
\mathcal{G}(\tilde u) \leqslant \liminf_{\ell \rightarrow \infty} \mathcal{G}_ {\ell} (U_\ell). 
\end{equation}
\end{lemma} 
\begin{proof}
From the  stability estimate, Proposition \ref{TD:EquicoercivityofG}, and the assumption on the boundedness 
of $ \mathcal{G}_{\ell} (U_\ell)  $  we conclude   that $\|\overline U_\ell\|_{L^2((0,T); H^2(\varOmega))} + \|\widehat U_\ell'\|_{L^2((0,T); L^2(\varOmega))} \leqslant C_1$
 are uniformly bounded. By the relative compactness in $L^2((0,T); L^2(\varOmega))$ we have (up to a subsequence not re-labeled)  the existence of $u_{(1)}$ 
 and $u_{(2)}$ such that 
\begin{equation}\label{proofGlowersemicontinuTDiscr}
L \overline  U_\ell \rightharpoonup L u_{(1)}\;\: \textrm{and} \;\: \widehat U_\ell' \rightharpoonup u_{(2)} '\;\: \textrm{weakly in} \;\: L^2((0,T); L^2(\varOmega))\, .
\end{equation}
Fix a  space-time test function $\varphi\in C^\infty _0 ,$ 
and  let $I_0$ be an appropriate interpolant into the piecewise constants in time functions. 
Then, \eqref{exact_integrals} implies 
\begin{equation}
	\begin{split}
  \Big | \int_0^T  \<  \widehat U _\ell  , \varphi ' \>   \d t     -& \int_0^T  \<  \overline U _\ell  , \varphi ' \>   \d t  \Big | \leqslant
 \Big | \int_0^T  \<  \widehat U _\ell -   \overline U _\ell , \varphi ' - I _0  \varphi ' \>   \d t       \Big | + 
 \Big | \int_0^T  \<  \widehat U _\ell -   \overline U _\ell ,   I _0  \varphi ' \>   \d t       \Big | \\
 &=\Big | \int_0^T  \<  \widehat U _\ell -   \overline U _\ell , \varphi ' - I _0  \varphi ' \>   \d t       \Big |.
	\end{split}
\end{equation}
By the uniform bound   $\|\overline U_\ell\|_{L^2((0,T); H^2(\varOmega))} + \|\widehat U_\ell\|_{L^2(0,T; H^2(\varOmega))} \leqslant C_1\, ,$ we infer that 
\begin{equation}
 \int_0^T  \<  \widehat U _\ell  , \varphi ' \>\,   \d t     - \int_0^T  \<  \overline U _\ell  , \varphi ' \>\,   \d t  \to 0, 
 \qquad \ell \to \infty\, , 
\end{equation}
and 
\begin{equation}
	\begin{split}
 &\int_0^T  \<  L\widehat U _\ell  , \varphi ' \>\,   \d t     - \int_0^T  \< L \overline U _\ell  , \varphi ' \>\,   \d t  \to 0, 
 \qquad \ell \to \infty\, . 
	\end{split}
\end{equation}
We can conclude, therefore, that  $u_{(1)} = u_{(2)} = \tilde u$, and thus  
\begin{equation}\label{proofGlowersemicontinuTDiscr2}
\begin{gathered}
  \widehat U_\ell' + L \overline  U_\ell -  \overline  f \rightharpoonup \tilde u' + L  \tilde u - f , 
 \qquad \ell \to \infty\, . 
\end{gathered}
\end{equation}
The convexity of $ \int_{\varOmega_T} | \cdot |^2 $ implies weak lower semicontinuity, that is,
\begin{equation}\label{proofGlowersemicontinuityeq2}
\int_{\varOmega_T} | \tilde u' + L \tilde u -f |^2 \, \d x\d t \leqslant \liminf_{\ell \rightarrow \infty}  \int_{\varOmega_T} |  \widehat U_\ell' + L \overline  U_\ell -  \overline  f|^2\, \d x\d t
\end{equation}
and therefore  the proof is complete. 
\end{proof}

\begin{lemma}[{$\limsup$ inequality}]\label{TD:limsup} We assume that the operator $\widehat \varPi _q$ can be represented as  
\begin{equation}
\label{eq:Lagr6}
\widehat  \varPi _q  v (t)=\sum_{i=0}^q {\tilde \ell}_{ni}(t) v(\tilde t_{ni}),
\quad t\in (t_n,t_{n+1}];
\end{equation}
see \eqref {eq:Lagr6}.  Let $ w \in H^1((0,T); L^2(\varOmega))\cap L^2((0,T); H^2(\varOmega)\cap H^1_0 (\varOmega)\, ) . $ Then, there exists a recovery 
 sequence   $\{w_\ell \},$ $ w_\ell \in V_\ell, $ such that $w_\ell \to w$ and 
\[ \mathcal{G}(w) = \lim_{\ell \rightarrow \infty} \mathcal{G}_{ \ell} (w_\ell) .\]
\end{lemma} 
\begin{proof}
 For $ w \in H^1((0,T); L^2(\varOmega))\cap L^2((0,T); H^2(\varOmega)\cap H^1_0 (\varOmega)\, ) $, we choose a smooth approximation 
   $(w_\delta) \subset C^\infty([0,T];  H^2(\varOmega)\cap H^1_0 (\varOmega)\, )$
    such that 
\begin{equation}
\begin{split}\label{w_delta_TD}
	\|&w -w_\delta\|_{H^1((0,T); L^2(\varOmega))\cap L^2((0,T); H^2(\varOmega))} \lesssim \delta   \quad \text{and}\\
 |& w_\delta ' |_{H^1((0,T); L^2(\varOmega))\cap L^2((0,T); H^2(\varOmega))} \lesssim \frac{1}{\delta} |w|_{H^1((0,T); L^2(\varOmega))\cap L^2((0,T); H^2(\varOmega))}.
	\end{split}
\end{equation}
We assign 
to each $\delta$  a discrete function $w_{\delta, \ell}\in V_\ell$  satisfying \eqref{w_ell_7_TD} and \eqref{w_ell_7_hSm_TD}. 
The recovery sequence will be  $\{w_{\delta , \ell} \},$ 
with appropriate  $\delta = \delta (\ell),$ and we shall show that  
\begin{equation}\label{GL2limit}
\mathcal{G}_{IE, \ell} (w_{\delta, \ell}) \rightarrow \mathcal{G}(w)\, .
\end{equation}
We split the error, 
\begin{equation}
	\begin{split}
 \|\widehat w_{\delta, \ell} '  + L \, \overline w_{\delta, \ell} - &  w' - L w   \|_{L^2((0,T); L^2(\varOmega))}
 \leqslant  \|\widehat w_{\delta, \ell} '  + L \, \overline w_{\delta, \ell} -   \widehat w_{\delta} '  -L \, \overline w_{\delta }   \|_{L^2((0,T); L^2(\varOmega))}\\
&+ \|\widehat w_{\delta } ' + L \, \overline w_{\delta } -  w_{\delta }' -  L w _{\delta } \|_{L^2((0,T); L^2(\varOmega))} 
+ \|  w_{\delta } ' + L \,   w_{\delta } -  w' -  L w  \|_{L^2((0,T); L^2(\varOmega))}\\
&=  : A_1 +A_2 +A_3\, .
	\end{split}
\end{equation}
Notice that since $\widehat  \varPi  v (t)=\sum_{i=0}^q {\tilde \ell}_{ni}(t) v(\tilde t_{ni}),$ we have $\sum_{i=0}^q {\tilde  \ell}_{ni}(t) =1$
and thus
\[\sum_{i=0}^q {\tilde  \ell}' _{ni}(t) =0\]
for all $n.$ Hence, 
\begin{equation}
	\begin{split}
  \|\widehat v  '        \|^2_{L^2((0,T); L^2(\varOmega))} 
& = \sum _{n=0}^{N-1} \, \int_{J_n}  \,  \big \|  \widehat v   '  \big  \|^2 _{ L^2(\varOmega)} \, \d t 
= \sum _{n=0}^{N-1} \, \int_{J_n}  \,  \big \|  \sum_{i=0}^q {\tilde \ell}_{ni} ' (t) v(\tilde t_{ni}) \big  \|^2 _{ L^2(\varOmega)} \, \d t \\
&  
= \sum _{n=0}^{N-1} \, \int_{J_n}  \,  \big \|  \sum_{i=1}^q {\tilde \ell} ' _{ni}(t) (v( \tilde t_{ni}) - v(\tilde t_{n0}) \, )\big  \|^2 _{ L^2(\varOmega)} \, \d t \\
&  
\leqslant \sum _{n=0}^{N-1} \, \int_{J_n}  \,   \sum_{i=1}^q \| {\tilde \ell} '_{ni} \| ^2_{L^\infty(J_n)}  \big \|  v(\tilde t_{ni}) - v(\tilde t_{n0}) \big  \|^2 _{ L^2(\varOmega)} \, \d t \\
     & \leqslant C\,  \sum _{n=0}^{N-1} \,    \, \frac 1 {k_n}  \sum_{i=1}^q    \big \|   \int_{\tilde t_{n0}} ^{\tilde t_{ni}}  v ' (t)  \, \d t  \big  \|^2 _{ L^2(\varOmega)} \,   \\
       & \leqslant C\,  \sum _{n=0}^{N-1} \,    \,    \sum_{i=1}^q      \int_{\tilde t_{n0}} ^{\tilde t_{ni}}  \big \| v ' (t) \big  \|^2 _{ L^2(\varOmega)} \, \d t  \,   \\
& \leqslant C\,  \sum _{n=0}^{N-1} \,   \,   \int_{J_n}  \big \| v  ' (t) \big  \|^2 _{ L^2(\varOmega)} \, \d t   
  =  C \|v  '     \|^2_{L^2((0,T); L^2(\varOmega))} \, .
	\end{split}
\end{equation}
Thus, for  $\theta _\ell (t) : = w_{\delta, \ell} (t)     -    w_{\delta}(t)\, ,$ we have 
\begin{equation}
	\begin{split}
  \|\widehat w_{\delta, \ell} '    -   \widehat w_{\delta} '     \|^2_{L^2((0,T); L^2(\varOmega))} 
=  \|\widehat \theta  _\ell  '     \|^2_{L^2((0,T); L^2(\varOmega))}
& \leqslant C\,   \|\theta  _\ell  '     \|^2_{L^2((0,T); L^2(\varOmega))} \, .
	\end{split}
\end{equation}
Next, we observe  that
\begin{equation}
	\begin{split}
  \|L \, \overline w_{\delta, \ell}   & -L \, \overline w_{\delta }     \| _{L^2((0,T); L^2(\varOmega))} 
  = \Big \{ \sum _{n=0}^{N-1} \, \int_{J_n}  \,  \big \| L \, \varPi_{q-1}  \theta  _\ell     \big  \|^2 _{ L^2(\varOmega)} \, \d t  \Big \}^{1/2} \\
 & \leqslant  \Big \{ \sum _{n=0}^{N-1} \, \int_{J_n}  \,  \big \| L \, \varPi_{q-1}  \theta  _\ell   - L \, P_0   \theta  _\ell     \big  \|^2 _{ L^2(\varOmega)} \, \d t  \Big \}^{1/2} 
  +\Big \{ \sum _{n=0}^{N-1} \, \int_{J_n}  \,  \big \| L \, P_0   \theta  _\ell     \big  \|^2 _{ L^2(\Omega)} \, \ d t  \Big \}^{1/2} \\
  &  \leqslant    \| L \, \varPi_{q-1}  \theta  _\ell   -  L\,    \theta  _\ell       \| _{L^2((0,T); L^2(\varOmega))} 
    +  \| L \,  \theta  _\ell   - L \, P_0   \theta  _\ell      \| _{L^2((0,T); L^2(\varOmega))}	+  \|L\, \theta  _\ell       \| _{L^2((0,T); L^2(\varOmega))} \, \\
     &  \leqslant C   k \,  \| L \,     \theta  _\ell  '      \| _{L^2((0,T); L^2(\varOmega))} 
    +    \|L\, \theta  _\ell       \| _{L^2((0,T); L^2(\varOmega))} \, , 
	\end{split}
\end{equation}
 where we have set  $k=\max_n k_n.$
 We have proved so far that  
\begin{equation}
A_1 \leqslant \|  \theta  _\ell  '     \| _{L^2((0,T); L^2(\varOmega))}  + \|L\, \theta  _\ell       \| _{L^2((0,T); L^2(\varOmega))} + k\,   \|L\, \theta  _\ell  '     \| _{L^2((0,T); L^2(\varOmega))} \, . 
\end{equation}
On the other hand, the approximation properties of ${\widehat \varPi}_{q}, \varPi_{q-1}  $ imply
  \begin{equation}
A_2 \leqslant C \, k \, \big [ 
\|  w_{\delta } '' \| _{L^2((0,T); L^2(\varOmega))}  + \| L \,   w_{\delta}'       \|_{L^2((0,T); L^2(\varOmega))}  \big ]\, .
\end{equation}
Using \eqref{w_ell_7_TD}, \eqref{w_ell_7_hSm_TD}, and \eqref{w_delta_TD}, 
we conclude, therefore, that  
\begin{equation}\label{final_eq}
A_1+ A_2 +A_3 
\leqslant   \, \beta_\ell (w_\delta) +  \frac k 
{\delta ^{m+1} }  \tilde \beta_\ell  \| w\| _{L^2((0,T); H^2(\varOmega))}     
+ C \frac k 
{\delta   }    \| w\| 
_{H^1((0,T); L^2(\varOmega))\cap L^2((0,T); H^2(\varOmega))}  + C \delta\, .
%
\end{equation}
The proof of   \eqref{GL2limit} is completed by suitably selecting  $\delta = \delta (\ell, k)$  in order that the right-hand side of \eqref{final_eq} converges to zero.   
\end{proof}

\begin{remark}[On the abstract assumptions on $\widehat  \varPi _q ,$ $\varPi _{q-1} $]  \label{assumptions_Pi}
	Notice that the assumption \eqref{eq:Lagr6} requires only that  
	$\widehat  \varPi _q $ is any interpolant onto piecewise polynomials of degree $q$ which preserves continuity at the nodes, 
	i.e., $\tilde t_{n0} = t_n$ and $\tilde t_{nq} = t_{n+1}.$ For certain methods, for example, for collocation methods where 
	the nodes include at least one end point of $[0,1]$ and for discontinuous Galerkin methods, one may select $\tilde t_{ni} = t_{ni} ,$
	$i=1,\dotsc , q,$ which may be convenient from computational perspective. 
	Assumption \eqref{exact_integrals}
	requires essentially that the interpolatory quadrature induced by $\varPi _{q-1} $ integrates exactly piecewise 
	polynomials of degree $q.$ This assumption is always satisfied by the methods considered in Section \ref{Se:4}. 
	\end{remark}

Next we shall combine the above results to show that the sequence of discrete minimizers $ (u_\ell) $ converges 
in $ L^2((0,T); H^1(\varOmega))  $ to the exact solution of our problem. We shall use the Aubin--Lions Lemma which  is an  
analogue of the Rellich--Kondrachov theorem in the parabolic case;   see, e.g., \cite{SZ}. 

\begin{lemma}[Aubin--Lions Lemma]\label{Aubin-LionsTheorem} 
Let $ B_0,B,B_1 $ be three Banach spaces, with $ B_0, B_1 $ reflexive. 
Suppose that $ B_0 $ is continuously embedded into $ B $, which is also continuously embedded into 
$ B _1$, and the embedding from $ B_0 $ into $ B $ is compact. For any given $ p_0,p_1 $ with $ 1 < p_0,p_1 < \infty $, let
\begin{align}
W = \lbrace v \: | \: v \in L^{p_0}((0,T);B_0) \;,\; v_t \in L^{p_1}((0,T); B_1) \rbrace.
\end{align}
Then, the embedding from $ W $ into $ L^{p_0}((0,T);B) $ is compact.
\end{lemma}

We are now ready to conclude the proof of the main result of this section.

\begin{theorem}[Convergence]\label{Thrm:TimeD} 
Let $ \mathcal{G},\; \mathcal{G}_{ \ell} $ be the energy functionals 
defined in \eqref{ParabFunctional} and \eqref{Functional_k}, respectively. 
Let $u$ be the exact solution of \eqref{ParabolicPDE} and let $ (u_\ell) , $ $u_\ell \in V_{\ell}, $ be a sequence 
of minimizers of $ \mathcal{G}_{ \ell}  $, i.e.,
\begin{equation}\label{minofTDEnergy}
\mathcal{G}_{\ell} (u_\ell) = \inf_{v_\ell \in W_\ell} \mathcal{G}_{ \ell} (v_\ell)\, .
\end{equation}
Then, 
\begin{equation}\label{TDConvOfDiscrMinE}
\hat u_\ell \rightarrow u,  \;\; \; \textrm{in} \;\: L^2((0,T); H^1(\varOmega))  ,\;\ \  
\end{equation}
where $\hat u_\ell ={\widehat \varPi}_{q}u_\ell \, .$  
%
%
\end{theorem}
\begin{proof}
Our assumptions imply that the solution  $u$  of \eqref{ParabolicPDE}  satisfies 
$u\in L^2((0,T); H^2(\varOmega)\cap H^1_0(\varOmega))\cap H^1((0,T); L^2(\varOmega)) ,$
and the elliptic regularity
\[ \| u\|  _ {L^2((0,T); H^2(\varOmega))} \leqslant C  \|Lu\|  _ {L^2((0,T); L^2(\varOmega))}   \]
is valid.  
Consider the sequence of minimizers $(u_\ell)\, .$ By their definition, 
\[ \mathcal{G}_{  \ell}  (u_\ell) \leqslant \mathcal{G}_{ \ell} (v_\ell), \qquad \text{for all } v_\ell \in V_\ell\, . \]
In particular,  
$ \mathcal{G}_{\ell}  (u_\ell) \leqslant \mathcal{G}_{ \ell}  (\tilde u_\ell),  $
where $\tilde u_\ell $ is the recovery sequence $w_{\delta, \ell} $ corresponding to $w=u$ constructed  in Lemma \ref{TD:limsup}. 
Since $ \mathcal{G}_{ \ell}  (\tilde u_\ell )$ converges to $ \mathcal{G}  (u )$,  we infer that the sequence $ \mathcal{G}_{ \ell}  (u_\ell)$
 is   uniformly  bounded. The stability of the discrete functional of Proposition \ref{TD:EquicoercivityofG}    yields   
 the uniform bound
\begin{align}
\|\overline u_\ell\|_{L^2((0,T); H^2(\varOmega))} +
\|\widehat u_\ell \|_{L^2((0,T); H^2(\varOmega))} + \|\widehat u_\ell'\|_{L^2((0,T); L^2(\varOmega))} \leqslant C  \, .
\end{align} 
  Applying the  Aubin--Lions Lemma,  we conclude the existence of  $   \tilde u\in L^2((0,T); H^1(\varOmega)) $ 
such that $ \widehat u _\ell   \rightarrow   \tilde u$ in 
$L^2((0,T); H^1(\varOmega))  $ up to a subsequence not re-labeled.
Furthermore, the arguments in Lemma \ref{TD:liminf} show that $L\tilde u \in L^2((0,T); L^2(\varOmega))\, .  $
Next we show that  $   \tilde u $ is the minimizer of $ \mathcal{G} ,$ and hence $   \tilde u= u.$ Indeed, let  
 $ w \in H^1((0,T); L^2(\varOmega))\cap L^2((0,T); H^2(\varOmega)\cap H^1_0(\varOmega)) $, and 
 let  $(w_\ell)$ be such that $w_\ell \to w$ and 
 \[ \mathcal{G}(w) = \lim_{\ell \rightarrow \infty} \mathcal{G}_{ \ell}  (w_\ell) .\]
   Therefore,  the $\liminf$ inequality 
and the fact that $ u_\ell $ are   minimizers of the discrete problems imply that
\begin{equation}
 \mathcal{G} (\tilde u)  
 \leqslant  \liminf_{\ell \rightarrow \infty}   \mathcal{G} _{  \ell}  (u_\ell)  
 \leqslant  \limsup_{\ell \rightarrow \infty} 
 \mathcal{G}_{  \ell}  (u_\ell) 
 \leqslant  \limsup_{\ell \rightarrow \infty} 
 \mathcal{G} _{\ell}  (w_\ell)  =  
 \mathcal{G}  (w) , 
\end{equation}
for all $ w \in H^1((0,T); L^2(\varOmega))\cap L^2((0,T); H^2(\varOmega)\cap H^1_0(\varOmega)). $  
Thus, $  \tilde u $ is   the  minimizer of $  \mathcal{G} ,$ and thus 
$   \tilde u= u$ and the entire sequence satisfies 
\[\hat u_\ell \rightarrow u,  \;\; \; \textrm{in} \;\: L^2((0,T); H^1(\varOmega))  .\;\ \   \]
 \end{proof}

\section{Discrete maximal parabolic $L^2$ regularity in Hilbert spaces}\label{Se:4}
We consider the discretization of differential equations satisfying the maximal parabolic
regularity property in Hilbert spaces by B-stable Runge--Kutta methods, the Lobatto IIIA methods, and 
Galerkin time-stepping methods. We establish discrete maximal parabolic $L^2$ regularity 
by the energy technique.

\subsection{An abstract initial value problem}\label{SSe:1.1}
We consider an initial value problem for a linear parabolic equation,
\begin{equation}
\label{ivp}
\left \{
\begin{aligned} 
&u' (t) +A u(t) =f(t), \quad 0<t< T,\\
&u(0)=u_0,
\end{aligned}
\right .
\end{equation}
in a Hilbert space $(H,(\cdot,\cdot)).$ We denote the induced norm by $|\cdot|, |v|=(v,v)^{1/2}, v\in H.$
We assume that $A$ is a coercive, self-adjoint, densely defined operator on $H,$ $u_0\in V:=\D(A^{1/2}),$ and
$f\in L^2((0,T);H).$

Taking the squares of the norms of both sides of the differential equation
in \eqref{ivp}, we have
\[|u'(s)|^2+|Au(s)|^2+ 2 (u'(s),Au(s))=|f(s)|^2,\]
i.e.,
\[|u'(s)|^2+|Au(s)|^2+\frac \d{\d s}|A^{1/2}u(s)|^2=|f(s)|^2.\]
Integration from $s=0$ to $s=t\in (0,T]$ yields the well-known   maximal $L^2$  regularity,
\begin{equation}
\label{max-reg}
\begin{split}
|A^{1/2}u(t)|^2+\|u'\|_{L^2((0,t);H)}^2+\|Au\|_{L^2((0,t);H)}^2&=|A^{1/2}u_0|^2
+\|f\|_{L^2((0,t);H)}^2\\
&\forall f\in L^2((0,t);H).
\end{split}
\end{equation}
 In other words, for vanishing initial value $u_0,$ the functions $u'$ and $Au$ are well defined 
and have the same regularity as their sum $u'+Au,$ that is,  the given forcing term $f;$
the sum of the norms and the norm of the sum are equivalent.

We refer to the lecture notes \cite{KuW} and to the review article \cite{DHP} for excellent accounts 
of the maximal $L^p$-regularity theory. Coercive elliptic differential operators on $L^s(\varOmega), 
1<s<\infty,$ with general boundary conditions possess  the maximal $L^p$-regularity  property; 
see  \cite{KuW},  \cite{DHP}, and references therein.
For maximal $L^p$-regularity properties of Runge--Kutta methods, see \cite{kovacs2016stable} and references
therein.

\subsection{The numerical methods}\label{SSe:1.2}
Recall that we are using a partition of the time interval  $[0,T],$ into subintervals $J_n:= (t_n,t_{n+1}], $ $n=0,\dotsc,N,$ and 
$k_n = |J_n|.$
Our results apply to arbitrary partitions.

For $s\in \N_0,$ we denote by  $\P(s)$ the space of polynomials of degree at most $s$
with coefficients in $\D(A),$ the domain of the operator $A,$ i.e., the elements $g$ of $\P(s)$  
are of the form 
\[g(t)= \sum_{j=0}^s  t^j w_j, \quad w_j\in  \D(A), \quad j=0,\dotsc, s.\]
With this notation, let $\V_k^{\text{c}} (s)$ and $\V_k^{\text{d}} (s)$ be the spaces of continuous 
and possibly discontinuous, respectively,  piecewise elements of $\P(s)$,
\[\begin{aligned}
&\V_k^{\text{c}} (s):=\{v\in C\big ([0,T];\D(A)\big ): v|_{J_n}\in \P(s), \ n=0,\dotsc, N-1\},\\
&\V_k^{\text{d}} (s):=\{v: [0,T]\to \D(A), \ v|_{J_n}\in \P(s), \ n=0,\dotsc, N-1\}.
\end{aligned}\]
The spaces $\H_k^{\text{c}} (s)$ and $\H_k^{\text{d}} (s)$ are defined analogously, with coefficients $w_j\in H$.

The numerical methods we consider here can be cast in the following abstract form: For $q\in \N,$ and
two given projection or interpolation operators $\varPi_{q-1} ,  {\widetilde \varPi}_{q-1} : C\big ([0,T];H\big )\to \H_k^{\text{d}} (q-1)$,
seek $\widehat U\in \V_k^{\text{c}} (q)$ 
satisfying the initial condition $\widehat U(0)=u_0$ and the pointwise equations
\begin{equation}
\label{eq:nm-abstr1}
\widehat U'(t)+\varPi_{q-1} A\widehat U(t)={\widetilde \varPi}_{q-1} f(t), \quad t\in (t_n,t_{n+1}], \quad n=0,\dotsc,N-1.
\end{equation}
%
Collocation methods as well as Galerkin time-stepping methods can be written in the form \eqref{eq:nm-abstr1};
see \cite{AMN3}. More precisely, $\widehat U$ is  the collocation approximation
and the continuous Galerkin approximation for these two classes of methods, and a suitable reconstruction 
of the solution for the discontinuous Galerkin method. Moreover, as we shall see, our key assumption 
\eqref{eq:nm-abstr2} in Theorem \ref{Th:max-reg:abstr} is satisfied for the Galerkin time-stepping methods 
as well as for some important collocation methods such as the Gauss, Radau IIA and Lobatto  IIIA methods. 
Furthermore, in Section \ref{SSe:Bstable}  we show maximal regularity estimates of  all algebraically 
stable Runge--Kutta methods.

 \begin{theorem}[Maximal $L^2$ regularity for  methods of the form \eqref{eq:nm-abstr1}]\label{Th:max-reg:abstr}
Assume that the operator $\varPi_{q-1}$ in \eqref{eq:nm-abstr1} is such that
\begin{equation}
\label{eq:nm-abstr2}
\int_{t_n}^{t_{n+1}}(\widehat U'(t),\varPi_{q-1}A\widehat U(t) )\, \d t \geqslant 
\int_{t_n}^{t_{n+1}}(\widehat U'(t), A\widehat U(t) )\, \d t =
\frac 12 \big (|A^{1/2}\widehat U(t_{n+1})|^2-|A^{1/2}\widehat U(t_n)|^2\big ).
\end{equation}
 Then, the method satisfies the following discrete analogue of the continuous maximal $L^2$ 
 regularity property \eqref{max-reg}
\begin{equation}
\label{eq:nm-abstr3}
|A^{1/2}\widehat U(t_m)|^2+\|\widehat U'\|_{L^2((0,t_m);H)}^2+\| \varPi_{q-1} A \widehat U\|_{L^2((0,t_m);H)}^2
\leqslant |A^{1/2}\widehat U(0)|^2+\|{\widetilde \varPi}_{q-1} f\|_{L^2((0,t_m);H)}^2
\end{equation}
for $m=1,\dotsc,N.$ 
 \end{theorem}
 
\begin{proof}
Taking the squares of the norms of both sides of the pointwise form \eqref{eq:nm-abstr1}
of the abstract numerical method, we infer that
\begin{equation*}
|\widehat U'(t)|^2+|\varPi_{q-1} A \widehat U(t)|^2+2(\widehat U'(t), \varPi_{q-1}  A \widehat U(t))=|{\widetilde \varPi}_{q-1} f(t)|^2, \quad t\in (t_n,t_{n+1}].
\end{equation*}
Integration over $J_n=(t_n,t_{n+1}]$ yields
\begin{equation}
\label{eq:nm-abstr4}
\begin{split}
\int_{t_n}^{t_{n+1}}|\widehat U'(t)|^2\, \d t&+\int_{t_n}^{t_{n+1}}| \varPi_{q-1} A \widehat U(t)|^2\, \d t
+2\int_{t_n}^{t_{n+1}}(\widehat U'(t),\varPi_{q-1} A\widehat U(t))\, \d t\\
&=\int_{t_n}^{t_{n+1}}|{\widetilde \varPi}_{q-1} f(t)|^2\, \d t.
\end{split}
\end{equation}
Now, utilizing  our assumption \eqref{eq:nm-abstr2}  in \eqref{eq:nm-abstr4}, we can estimate the last term on
the left-hand side and obtain
\begin{equation}
\label{eq:nm-abstr5}
\begin{split}
|A^{1/2}\widehat U(t_{n+1})|^2&+\int_{t_n}^{t_{n+1}}|\widehat U'(t)|^2\, \d t+\int_{t_n}^{t_{n+1}}|\varPi_{q-1} A \widehat U(t)|^2\, \d t\\
&\leqslant |A^{1/2}\widehat U(t_n)|^2+\int_{t_n}^{t_{n+1}}|{\widetilde \varPi}_{q-1} f(t)|^2\, \d t.
\end{split}
\end{equation}
Summation over $n$ from $n=0$ to $n=m-1\leqslant N-1,$ yields the asserted maximal regularity 
estimate \eqref{eq:nm-abstr3}.
\end{proof}

\subsection{Galerkin time-stepping methods}\label{SSe:Gal}

\subsubsection{Continuous Galerkin methods}\label{SSe:cG}
For $q\in \N,$ with starting value $U(0)=u_0,$ we  consider the discretization 
of the initial value problem \eqref{ivp}  by the 
\emph{continuous Galerkin method} cG$(q)$, i.e., 
we seek $\widehat U\in \V_k^{\text{c}} (q)$ such that
\begin{equation}
\label{eq:cG1}
\int_{J_n}   \big( (\widehat  U' ,v )  + ( A\widehat U ,v ) \big) \, \d t  
= \int_{J_n} (f,v)\, \d t \quad \forall v \in \P(q-1)
\end{equation}
for $n=0,\dotsc,N-1$.  Denoting by $P_{q-1}$ the piecewise $L^2$-projection onto $\H_k^{\text{d}}(q-1)$
and using the fact that $\widehat  U' \in \V_k^{\text{d}} (q-1)$, we see that the  \emph{pointwise form} of
\eqref{eq:cG1} is
\begin{equation}
\label{eq:cG2}
   \widehat U' +  P_{q-1} A \widehat U = P_{q-1}f ,
\end{equation}
which is \eqref{eq:nm-abstr1} with $\varPi_{q-1} :={\widetilde \varPi}_{q-1}:=P_{q-1}.$

In this case, since $\widehat  U' \in \V_k^{\text{d}} (q-1)$, we have
\[\int_{t_n}^{t_{n+1}} (\widehat U'(t),A\widehat U(t)-P_{q-1}  A\widehat U(t) )\, \d t=0\]
and see that the key assumption  \eqref{eq:nm-abstr2} holds true as an equality.
It will be useful to observe the following relation
%
\begin{equation}
\label{eq:GL1}
\int_{t_n}^{t_{n+1}}(\varphi ,I_{G, q-1} W(t))\, \d t
= \int_{t_n}^{t_{n+1}}(\varphi , W(t))\, \d t \quad \forall \varphi \in \P(q-1)\, , 
\end{equation}
where $W \in  \P(q)$ and   $I_{G, q-1} v \in \P(q-1)$ denotes the interpolant of $v$ at the $q $ Gauss points of $J_n;$
indeed, the integrand $(\varphi ,W-I_{G, q-1} W)\in \P_{2q-1}$ is integrated exactly by the Gauss quadrature formula 
with $q$ nodes and it vanishes at these nodes.
Therefore, 
\[P_{q-1} | _ {\P(q)} = I_{G, q-1} .\]

Therefore, as an immediate consequence of Theorem \ref{Th:max-reg:abstr},
we have the following maximal $L^2$ regularity of cG methods:

 \begin{proposition}[Maximal $L^2$ regularity of cG methods]\label{Co:max-reg:cG}
 The cG methods satisfy the following analogue of the continuous maximal $L^2$ 
 regularity property \eqref{max-reg}
\begin{equation}
\label{eq:cG3}
\begin{split}
&|A^{1/2}\widehat U(t_m)|^2+\|\widehat U'\|_{L^2((0,t_m);H)}^2+\| P_{q-1} A \widehat U\|_{L^2((0,t_m);H)}^2   \\
&=|A^{1/2}\widehat U(t_m)|^2+\|\widehat U'\|_{L^2((0,t_m);H)}^2+\| I_{G, q-1} A \widehat U\|_{L^2((0,t_m);H)}^2\\
&= |A^{1/2}\widehat U(0)|^2+\|P_{q-1} f\|_{L^2((0,t_m);H)}^2\\
\end{split}
\end{equation}
for $m=1,\dotsc,N.$ 
 \end{proposition}

Since, obviously,
\begin{equation}
\label{eq:cG+dG}
\int_{t_n}^{t_{n+1}}|P_{q-1} f(t)|^2\, \d t\leqslant \int_{t_n}^{t_{n+1}}|f(t)|^2\, \d t
\implies \|P_{q-1} f\|_{L^2((0,t_m);H)}^2\leqslant  \|f\|_{L^2((0,t_m);H)}^2,
\end{equation}
\eqref{eq:cG3} yields also the estimate
\begin{equation}
\label{eq:cG4}
|A^{1/2}\widehat U(t_m)|^2+\|\widehat U'\|_{L^2((0,t_m);H)}^2+\| P_{q-1} A \widehat U\|_{L^2((0,t_m);H)}^2\leqslant |A^{1/2}\widehat U(0)|^2+\|f\|_{L^2((0,t_m);H)}^2
\end{equation}
for $m=1,\dotsc,N.$

\subsubsection{Discontinuous Galerkin methods}\label{SSe:dG}

For $q\in \N,$ with starting value $U(0)=u_0,$ we  consider the discretization 
of the initial value problem \eqref{ivp}  by the 
\emph{discontinuous Galerkin method} dG$(q-1)$, i.e., 
we seek $U\in \V_k^{\text{d}} (q-1)$ such that
\begin{equation}
\label{dg}
\int_{J_n}   \big( (U' ,v )  + ( AU ,v ) \big) \, \d t  
+ ( U_n^{+}-U_n, v_n^{+})
= \int_{J_n} (f,v)\, \d t \quad \forall v \in \P(q-1)
\end{equation}
for $n=0,\dotsc,N-1$.  As usual, we use the notation $v_n:=v(t_n),$  $v_n^{+}:=\lim_{s\searrow 0} v(t_n+s)$.

Following \cite{MakrN2006posteriori},  we define the \emph{reconstruction} $\wU$ of the dG approximation $U,$
the analogue of the collocation approximation,  by extended interpolation at the Radau nodes
 $t_{ni}=t_n+c_i k_n, 0<c_1<\dotsb<c_q=1,$
\[\widehat U ( t_{ni})= U( t_{ni}), \quad i=0, \dotsc ,q \quad (U(t_{n0})=U_n).\]
The reconstruction satisfies the relations 
\begin{equation}
\label{wU_def}
\begin{split}
&\wU_n^+ = U_n,\\
&\int_{J_n}    ( \wU' ,v )\,\d t = \int_{J_n}    ( U' ,v )   \, \d t
+ ( U_n^+-U_n, v_n^+)  \quad \forall v \in \P(q-1) .
\end{split}
\end{equation}
Consequently, we can  reformulate the discontinuous Galerkin method \eqref{dg} as
\begin{equation}
\label{dgwU}
\int_{J_n} \big(   ( \wU' ,v )  + ( AU  ,v ) \big)\, \d t
= \int_{J_n} (f,v) \, \d t  \quad \forall v \in \P(q-1) .
\end{equation}
Denoting again by $P_{q-1}$ the piecewise $L^2$-projection onto $\H_k^{\text{d}}(q-1)$, 
we see that the \emph{pointwise form} of \eqref{dgwU} is
\begin{equation}
\label{eq:dG1}
   \widehat U' +  AU = P_{q-1}f ,
\end{equation}
i.e.,
\begin{equation}
\label{eq:dG2}
   \widehat U' +  I_{q-1} A \widehat U = P_{q-1}f ,
\end{equation}
which is \eqref{eq:nm-abstr1} with $\varPi_{q-1}:= I_{q-1}$ and  ${\widetilde \varPi}_{q-1}:=P_{q-1}.$

 
Let us now see that our key assumption  \eqref{eq:nm-abstr2} is satisfied also in this case, i.e., that
\begin{equation}
\label{eq:dG4} 
\int_{t_n}^{t_{n+1}}\big (\widehat U'(t),A(\widehat U(t) -I_{q-1}\widehat U(t))\big )\, \d t\leqslant 0.
\end{equation}

It is advantageous to reformulate \eqref{eq:dG4} in the form
\begin{equation}
\label{eq:dG5}
\begin{split}
&\int_{t_n}^{t_{n+1}}\big ((\widehat U-I_{q-1}\widehat U)'(t),A(\widehat U(t) -I_{q-1}\widehat U(t))\big )\, \d t\\
&+\int_{t_n}^{t_{n+1}}\big ((I_{q-1}\widehat U)'(t),A(\widehat U(t) -I_{q-1}\widehat U(t))\big )\, \d t\leqslant 0.
\end{split}
\end{equation}
Now, the integrand $\tilde \pi:=\big ((I_{q-1}\widehat U)'(\cdot),A(\widehat U(\cdot) -I_{q-1}\widehat U(\cdot))\big )$
in the second integral in \eqref{eq:dG5} is a polynomial of degree at most $2q-2;$ therefore, $\tilde \pi$ is integrated 
exactly by the Radau quadrature formula with $q$ nodes. Furthermore, $\tilde \pi$
vanishes at the quadrature nodes $t_{n1},\dotsc,t_{nq}$. Thus, the second integral 
vanishes, and \eqref{eq:dG4} can be written in the form
\begin{equation*}
\int_{t_n}^{t_{n+1}}\big (\widehat U'(t)-(I_{q-1}\widehat U)'(t),A(\widehat U(t) -I_{q-1}\widehat U(t))\big )\, \d t\leqslant 0,
\end{equation*}
i.e.,
\begin{equation*}
\frac 12 \int_{t_n}^{t_{n+1}}\frac {\d}{\d t} |A^{1/2}(\widehat U(t) -I_{q-1}\widehat U(t))|^2\, \d t\leqslant 0,
\end{equation*}
that is, since $t_{nq}=t_{n+1},$
\begin{equation*}
-\frac 12 |A^{1/2}(\widehat U(t_n) -I_{q-1}\widehat U(t_{n+}))|^2\leqslant 0,
\end{equation*}
which is obviously valid. Therefore, \eqref{eq:dG4} is valid.

In view of \eqref{eq:dG4}, as an immediate consequence of Theorem \ref{Th:max-reg:abstr},
we have the following maximal $L^2$ regularity of dG methods:

  \begin{proposition}[Maximal $L^2$ regularity of dG methods]\label{Th:max-reg:dG}
 The dG methods satisfy the following analogue of the continuous maximal $L^2$ 
 regularity property \eqref{max-reg}
\begin{equation}
\label{eq:dG1-mr}
|A^{1/2}\widehat U(t_m)|^2+\|\widehat U'\|_{L^2((0,t_m);H)}^2+\| AU\|_{L^2((0,t_m);H)}^2\leqslant |A^{1/2}\widehat U(0)|^2+\|P_{q-1} f\|_{L^2((0,t_m);H)}^2
\end{equation}
for $m=1,\dotsc,N.$ 
 \end{proposition}

In view of \eqref{eq:cG+dG},  \eqref{eq:dG1} yields also the estimate
\begin{equation}
\label{eq:dG6}
|A^{1/2}\widehat U(t_m)|^2+\|\widehat U'\|_{L^2((0,t_m);H)}^2+\| AU\|_{L^2((0,t_m);H)}^2\leqslant |A^{1/2}\widehat U(0)|^2+\|f\|_{L^2((0,t_m);H)}^2
\end{equation}
for $m=1,\dotsc,N.$

\subsection{Collocation Runge--Kutta methods}\label{SSe:1.3}
For $q\in \N,$ let $0\leqslant c_1<\dotsb<c_q\leqslant 1$ denote the collocation nodes.
The collocation approximation $\widehat U\in \V_k^{\text{c}} (q)$ satisfies the initial condition $\widehat U(0)=u_0$ 
as well as the collocation conditions
\begin{equation}
\label{eq:coll-1}
\widehat U'(t_{ni})+A\widehat U(t_{ni})=f(t_{ni}), \quad i=1,\dotsc,q, \quad n=0,\dotsc,N-1.
\end{equation}
Here, we assumed that $f(t)\in H$ for $t\in (0,T].$
Thus,   \cite{AMN2} and \cite{AMN3}, if we let $I_{q-1} : C\big ([0,T];H\big )\to \H_k^{\text{d}} (q-1)$
denote the interpolation operator at the collocation nodes $t_{ni}, i=1,\dotsc,q, n=0,\dotsc,N-1,$
and use the fact that $\widehat U'\in \V_k^{\text{d}} (q-1),$ we can write \eqref{eq:coll-1} in \emph{pointwise form} as 
\begin{equation}
\label{eq:coll-2}
\widehat U'(t)+I_{q-1} A\widehat U(t)=I_{q-1} f(t), \quad t\in (t_n,t_{n+1}], \quad n=0,\dotsc,N-1,
\end{equation}
which is \eqref{eq:nm-abstr1} with $\varPi_{q-1}:= {\widetilde \varPi}_{q-1}:=I_{q-1}.$
The interpolants $U:=I_{q-1} \widehat U$ and $I_{q-1} f$ are elements of $\V_k^{\text{d}} (q-1)$ and $\H_k^{\text{d}} (q-1),$
respectively, and thus, in general, for positive $c_1,$ discontinuous at the nodes $t_0,\dotsc,t_{N-1}.$

The corresponding $q$-stage Runge--Kutta method is specified by the coefficients
\begin{equation}
\label{RK-coef}
a_{ij} =\int _0^{c_i} \ell_j (\tau )\, \d \tau, 
\quad b_{i} =\int _0^{1} \ell_i(\tau )\, \d \tau,
\quad   i,j=1,\dotsc,q;
\end{equation}
here, $\ell_1,\dotsc, \ell_q\in \P_{q-1}$ are the Lagrange polynomials for the
collocation nodes $c_1,\dotsc,c_q,$ $\ell_i(c_j)$ $=\delta_{ij}, i,j=1,\dotsc,q.$
In other words, the \emph{stage order} of the Runge--Kutta method is $q.$ 

With starting value $U_0=u_0,$ we  now consider the discretization of the initial value problem \eqref{ivp}  
by the $q$-stage Runge--Kutta method \eqref{RK-coef}: we recursively define approximations 
$U_\ell\in \D(A)$ to the nodal values $u(t_\ell)$, as well as internal approximations $U_{\ell i}\in \D(A)$ 
to the intermediate  values $u(t_{\ell i}),$ by
\begin{equation}\label{eq:RK1}
\left \{
\begin{alignedat}{2}
&U_{ni}=U_n- k_n \sum_{j=1}^q a_{ij} \big (A U_{nj} -f(t_{nj})\big ),\quad &&i=1,\dotsc,q,\\
&U_{n+1} =U_n- k_n \sum_{i=1}^q b_i \big (A U_{ni} -f(t_{n i})\big ), \quad &&
\end{alignedat}
\right . 
\end{equation}
$n=0,\dotsc,N-1.$ Here, we assumed that $f(t)\in H$ for $t\in (0,T].$

It is well known that  the collocation and Runge--Kutta methods \eqref{eq:coll-1} and \eqref{eq:RK1},
respectively, are equivalent in the sense that they yield the same approximations at the nodes
and at the intermediate nodes, i.e.,
\begin{equation}\label{eq:equivalent}
\begin{alignedat}{2}
&\widehat U(t_n)=U_n, \quad &&n=1,\dotsc,N,\\
&\widehat U(t_{ni})=U_{ni}, \quad &&i=1,\dotsc,q, \ \, n=0,\dotsc,N-1.
\end{alignedat}
\end{equation}

\subsubsection{Maximal regularity of Gauss and Radau IIA methods}\label{SSe:Ga-Ra}
We treat this case separately for various reasons: (i) These two families of Runge--Kutta
methods are particularly interesting and popular for parabolic equations. (ii) The methods
satisfy our key assumption \eqref{eq:nm-abstr2} and, consequently, 
discrete analogues of the continuous maximal $L^2$ regularity property
\eqref{max-reg}, with inequality in the place of the equality for the Radau IIA methods. 
(iii) The proofs are short and elegant, immediate consequences of the abstract result
in Theorem \ref{Th:max-reg:abstr}.

We have already seen that assumption  \eqref{eq:nm-abstr2} is satisfied for the Radau IIA methods;
see \eqref{eq:dG2} and recall that $I_{q-1}$ there is the interpolation operator at the Radau nodes
$t_{n1},\dotsc,t_{nq}.$ 

Now, we will see  that the Gauss method satisfies \eqref{eq:dG4} as an equality, 
with  $I_{q-1}$ now, of course,  the interpolation operator at the Gauss nodes
$t_{n1},\dotsc,t_{nq}.$ Indeed,  the integrand $\pi:=\big (\widehat U'(\cdot),A(\widehat U(\cdot) -I_{q-1}\widehat U(\cdot))\big )$
in \eqref{eq:dG4} is a polynomial of degree at most $2q-1;$ therefore, $\pi$ is integrated 
exactly by the Gauss quadrature formula with $q$ nodes. Furthermore, $\pi$
vanishes at the quadrature nodes $t_{n1},\dotsc,t_{nq}$. Thus, 
\eqref{eq:dG4} holds as an equality in this case.

In view of \eqref{eq:dG4}, as an immediate consequence of Theorem \ref{Th:max-reg:abstr},
we have the following maximal $L^2$ regularity property for Gauss and Radau IIA methods:

 \begin{proposition}[Maximal $L^2$ regularity of Gauss and Radau IIA methods]\label{Co:max-reg:G+Rad1}
 The Gauss and Radau IIA methods satisfy the exact discrete analogues of the continuous maximal $L^2$ 
 regularity property \eqref{max-reg}, namely,
\begin{equation}
\label{eq:Ga-Ra1}
|A^{1/2}\widehat U(t_m)|^2+\|\widehat U'\|_{L^2((0,t_m);H)}^2+\|I_{q-1} A\widehat U\|_{L^2((0,t_m);H)}^2\leqslant |A^{1/2}\widehat U(0)|^2+\|I_{q-1} f\|_{L^2((0,t_m);H)}^2
\end{equation}
for $m=1,\dotsc,N,$ with equality for the Gauss methods.
 \end{proposition}

Let us now give an alternative form of \eqref{eq:Ga-Ra1}.

 \begin{proposition}[Alternative form of the maximal $L^2$ regularity of Gauss and Radau IIA methods]\label{Cor:max-reg:G+Rad2}
Let $0<c_1,\dotsc,c_q\leqslant 1$ and $b_1,\dotsc,b_q$ be the nodes and the weights of the Gauss and Radau
quadrature formulas in the interval $[0,1],$ respectively.
Then, the Gauss and Radau IIA methods satisfy the maximal $L^2$ regularity property 
\begin{equation}
\label{eq:Ga-Ra6}
\begin{split}
|A^{1/2}U_m|^2&+ \sum_{n=0}^{m-1}  k_n  \sum_{i=1}^q b_i |\widehat U'(t_{ni})|^2
+ \sum_{n=0}^{m-1} k_n \sum_{i=1}^q b_i |AU_{ni}|^2\\
&\leqslant |A^{1/2}U_0|^2+ \sum_{n=0}^{m-1} k_n  \sum_{i=1}^q b_i |f(t_{ni})|^2
\end{split}
\end{equation}
for $m=1,\dotsc,N,$ with equality for the Gauss methods. Here, in the case of the Radau IIA methods,
$\widehat U'(t_{nq})$ stands for the left-hand derivative at $t_{n+1}, \lim_{t\nearrow t_{n+1}}\widehat U'(t).$
 \end{proposition}

\begin{proof}
Obviously, $|\widehat U'|^2, |I_{q-1}A\widehat U|^2,$ and  $|I_{q-1}f|^2$ are integrated exactly by both the Gauss and Radau
quadrature formulas in each subinterval $[t_n,t_{n+1}]$ as polynomials of degree at most $2q-2.$ Consequently,
for instance,
\[\int_{t_n}^{t_{n+1}}|\widehat U'(t)|^2\, \d t=k _n\sum_{i=1}^q b_i |\widehat U'(t_{ni})|^2,\]
and  \eqref{eq:Ga-Ra1} can be equivalently written in the form  \eqref{eq:Ga-Ra6}.
\end{proof}

\subsubsection{Maximal regularity of  Lobatto IIIA methods}\label{SSe:Lobatto} 
Here, we focus on the Lobatto IIIA methods, which are A-stable but are not B-stable.
So, for $q\in \N,$ let $0=c_1<\dotsb<c_q= 1$ denote the Lobatto nodes;
then, the collocation approximation $\widehat U\in \V_k^{\text{c}} (q)$ satisfies 
the initial condition $\widehat U(0)=u_0$ 
as well as the collocation conditions
\begin{equation}
\label{eq:Lob1}
\widehat U'(t_{ni})+A\widehat U(t_{ni})=f(t_{ni}), \quad i=1,\dotsc,q, \quad n=0,\dotsc,N-1.
\end{equation}
We assumed that $f(t)\in H$ for $t\in (0,T].$ In this case, the pointwise form of the method is again
\begin{equation}
\label{eq:Lob2}
\widehat U'(t)+I_{q-1} A\widehat U(t)=I_{q-1} f(t), \quad t\in (t_n,t_{n+1}], \quad n=0,\dotsc,N-1;
\end{equation}
compare to \eqref{eq:coll-2}. Notice, however, the important fact that the interpolants $\widetilde U:=I_{q-1} \widehat U$ 
and $I_{q-1} f$ are now elements of $\V_k^{\text{c}} (q-1)$ and $\H_k^{\text{c}} (q-1),$ 
respectively, since $c_0=0$ and $c_q=1,$ and \emph{thus are continuous functions at the   nodes} $t_0,\dotsc,t_{N-1}.$ 

Now, we claim that the interpolant  $\widetilde U=I_{q-1} \widehat U\in \V_k^{\text{c}} (q-1)$  of the Lobatto collocation approximation
$\widehat U$  is the solution of a \emph{modified} continuous Galerkin (cG) method  in $\V_k^{\text{c}} (q-1),$ namely,
$\widetilde U\in \V_k^{\text{c}} (q-1)$ is such that
\begin{equation}
\label{eq:cG+Lob}
\int_{t_n}^{t_{n+1}}  \big( (\widetilde U' ,v )  + ( A\widetilde U ,v ) \big) \, \d t  
= \int_{J_n} (I_{q-1}  f,v)\, \d t \quad \forall v \in \P(q-2)
\end{equation}
for $n=0,\dotsc,N-1$, with the modification consisting in the fact that the forcing term
$f$ on the right-hand side has been replaced by its interpolant   $I_{q-1}  f.$
Compare to \eqref{eq:cG1} and notice that  \eqref{eq:cG+Lob} is a modification of the
cG$(q-1)$ rather than the cG$(q)$ method.


Now, \eqref{eq:Lob2} implies 
\begin{equation}
\label{eq:Lob3}
\int_{t_n}^{t_{n+1}}   \big( (\widehat U' ,v )  + ( A\widetilde U ,v ) \big) \, \d t  
= \int_{t_n}^{t_{n+1}} (I_{q-1}f,v)\, \d t \quad \forall v \in \P(q-2)
\end{equation}
for $n=0,\dotsc,N-1$.  In view of the fact that  $\widetilde U=I_{q-1} \widehat U,$ 
 \eqref{eq:cG+Lob} follows immediately from \eqref{eq:Lob2} provided
\begin{equation}
\label{eq:Lob4}
\int_{t_n}^{t_{n+1}}    ( \wU'-(I_{q-1} \widehat U )' ,v )\,\d t = 0  \quad \forall v \in \P(q-2) .
\end{equation}
Since $(I_{q-1} \widehat U ) (t_m)= \widehat U  (t_m) ,$ integrating by parts, we can rewrite  
\eqref{eq:Lob4} as
\begin{equation}
\label{eq:Lob5}
\int_{t_n}^{t_{n+1}}    ( \wU - I_{q-1} \widehat U,v' )\,\d t = 0
   \quad \forall v \in \P(q-2) .
\end{equation}
Now,  the integrand $\pi:=\big (\wU(\cdot)  - I_{q-1} \widehat U(\cdot) , v'(\cdot)  )$
in \eqref{eq:Lob5} is a polynomial of degree at most $2q-3;$ therefore, $\pi$ is integrated 
exactly by the Lobatto quadrature formula with $q$ nodes. Furthermore, $\pi$
vanishes at the quadrature nodes $t_{n1},\dotsc,t_{nq}$. Thus, 
\eqref{eq:Lob5}, and hence also \eqref{eq:Lob4}, is indeed valid.

It is straightforward now to apply the maximal regularity estimate for the continuous Galerkin method to conclude:

 \begin{proposition}[Maximal $L^2$ regularity of Lobatto IIIA methods]\label{Co:max-reg:GLob}
Let $I_{G, q-2} v \in \P(q-2)$ denote the interpolant of $v$ at the $q -1$ Gauss points of $J_n.$
 The Lobatto IIIA methods satisfy the following analogue of the continuous maximal $L^2$ 
 regularity property \eqref{max-reg}
\begin{equation}
\label{eq:GLo1}
\begin{split}
&|A^{1/2} U_m|^2+\|\widetilde  U'\|_{L^2((0,t_m);H)}^2+\| I_{G, q-2} A \widetilde U\|_{L^2((0,t_m);H)}^2   \\
&= |A^{1/2}U_0|^2+\|P_{q-2}I_{q-1} f\|_{L^2((0,t_m);H)}^2\\
\end{split}
\end{equation}
and 
\begin{equation}
\label{eq:GLo2}
\begin{split}
&|A^{1/2}U_m|^2+\|\widetilde  U'\|_{L^2((0,t_m);H)}^2+\| I_{G, q-2} A \widetilde U\|_{L^2((0,t_m);H)}^2   \\
&\leqslant  |A^{1/2} U_0|^2+\| I_{q-1} f\|_{L^2((0,t_m);H)}^2\\
\end{split}
\end{equation}
for $m=1,\dotsc,N.$ 
 \end{proposition}

\begin{remark}[The Trapezoidal Method]\label{Re:Trapezoidal}It is interesting to note that for the trapezoidal method, 
one is able to prove directly the estimate, 
\begin{equation}
\label{eq:Trap}
\begin{split}
|A^{1/2} U_m|^2&+ \sum_{n=0}^{m-1}  k_n  |\frac {U_{n+1}-U_{n }} {k_n} |^2
+ \sum_{n=0}^{m-1} {k_n}  |\frac {AU_{n+1}+AU_{n}} 2 | ^ 2 \\
& =|A^{1/2} U_0|^2+ \sum_{n=0}^{m-1} {k_n}  |\frac {f(t_{n+1})+f(t_{n })} 2 |^2
\end{split}
\end{equation}
which is identical to \eqref{eq:GLo1} for $q=2 .$ 
Hence, \eqref{eq:GLo1} is a natural but  non obvious generalization of \eqref{eq:Trap}. 
\end{remark}

\begin{remark}[Alternative version of \eqref{eq:Trap}]\label{Re:Trapezoidal2}
Maximal regularity estimates of the form
\begin{equation}
\label{eq:Trap-KLL}
 \sum_{n=0}^{m-1}  k  |\frac {U_{n+1}-U_n} {k} |^2
+ \sum_{n=1}^m k  |AU_n |^ 2 
\leqslant C \sum_{n=0}^m k  |f(t_n) |^2
\end{equation}
for the trapezoidal method for constant time steps, and for $U_0=0,$ are 
established in  \cite[Theorem 3.2]{kovacs2016stable}, actually for any 
$p\in (1,\infty)$ and for general UMD Banach spaces.
 An advantage of \eqref{eq:Trap} is that it holds as an equality
 and it is valid for arbitrary partitions.
 High order Lobatto IIA methods are not included in the analysis in  \cite[Theorem 3.2]{kovacs2016stable}.
\end{remark}

\begin{remark}[Equivalence between $\widehat U$ and $\widetilde U$]\label{Re:Lobatto}
If the Lobatto collocation approximation $\widehat U$ is available in a subinterval $\bar J_n=[t_n,t_{n+1}],$
then $\widetilde U\in \P(q-1)$ is obviously the interpolant of $\widehat U$ at the Lobatto nodes,
\begin{equation}
\label{eq:equivalence1}
\widetilde U(t_{ni})= \widehat U(t_{ni}),\quad i=1,\dotsc,q. 
\end{equation}
Conversely,  it $\widetilde U$ is available in $\bar J_n,$ then  the Lobatto collocation approximation $\widehat U\in \P(q)$
is uniquely determined by the interpolation conditions
\begin{equation}
\label{eq:equivalence2}
\widehat  U(t_{ni})= \widetilde U(t_{ni}),\quad i=1,\dotsc,q,
\quad \widehat  U'(t_n)=-A\widetilde  U(t_n)+f(t_n). 
\end{equation}
\end{remark}

\subsubsection{Maximal regularity of  algebraically stable Runge--Kutta methods}\label{SSe:Bstable}

Our main assumption on the Runge--Kutta method is that it is B-stable.
Since the collocation nodes $c_1,\dotsc,c_q$ are pairwise distinct,
it is well known that the B-stability is equivalent to the  \emph{algebraic stability}
 of the method; in other words, 
the weights $b_1,\dotsc,b_q$ are nonnegative and the $q\times q$ symmetric matrix $M$ 
with entries $m_{ij}:=b_ia_{ij}+b_ja_{ji}-b_ib_j, i,j=1,\dotsc,q,$
is positive semidefinite,
\begin{equation}\label{eq:alg-stab}
b_i\geqslant 0,\quad  i=1,\dotsc,q,\quad\text{and}\quad M\in\R^{q,q}\  \ \text{is positive semidefinite}.
\end{equation}
Notice also that, in the case of positive $c_1,$  the coefficient matrix $\AA:=(a_{ij})_{i,j=1,\dotsc,q}\in \R^{q,q}$ of
the Runge--Kutta method is invertible since the collocation nodes $c_1,\dotsc,c_q$ are pairwise distinct
and positive.

In the following  calculations we closely follow the proof that algebraically stable methods are B-stable.
With  $\varphi_j:= -k_n\big (A U_{nj} -f(t_{nj})\big )=-k_n\widehat U'(t_{nj})\in \D(A)$ (see \eqref{eq:coll-1} and  \eqref{eq:equivalent}),
we apply $A^{1/2}$  to  \eqref{eq:RK1} and write  it in the form
\begin{equation}\label{eq:Bstab1}
\left \{
\begin{alignedat}{2}
&A^{1/2}U_{ni}=A^{1/2}U_n+  A^{1/2}\sum_{j=1}^q a_{ij} \varphi_j,\quad &&i=1,\dotsc,q,\\
&A^{1/2}U_{n+1} =A^{1/2}U_n+  A^{1/2}\sum_{i=1}^q b_i \varphi_i. \quad &&
\end{alignedat}
\right . 
\end{equation}
We take the squares of the norms of both sides of the second relation of \eqref{eq:Bstab1},
and obtain 
\begin{equation}
\label{eq:Bstab2} 
|A^{1/2}U_{n+1}|^2 = |A^{1/2}U_n|^2 +2 \sum _{i=1}^q b_i(A^{1/2}\varphi_i, A^{1/2}U_n)
+ \sum _{i,j=1}^q b_ib_j( A^{1/2}\varphi_i,A^{1/2}\varphi_j) . 
\end{equation}
%
Using the first relations of \eqref{eq:Bstab1} we get
\[ \sum _{i=1}^q b_i(A^{1/2}\varphi_i, A^{1/2}U_n)=
\sum _{i=1}^q b_i(A^{1/2}\varphi_i, A^{1/2}U_{ni})-
   \sum _{i,j=1}^q b_ia_{ij} (A^{1/2}\varphi_i, A^{1/2}\varphi_j), \]
and \eqref{eq:Bstab2} leads to
\begin{equation*}
|A^{1/2}U_{n+1}|^2 = |A^{1/2}U_n|^2 - \sum _{i,j=1}^qm_{ij}( A^{1/2}\varphi_i,A^{1/2}\varphi_j)
+2 \sum _{i=1}^q b_i(\varphi_i, AU_{ni}), 
\end{equation*}
and, in view of the positive semidefiniteness of the matrix $M,$ to
\begin{equation}
\label{eq:Bstab3}
|A^{1/2}U_{n+1}|^2 \leqslant |A^{1/2}U_n|^2+2 \sum _{i=1}^q b_i(\varphi_i, AU_{ni}). 
\end{equation}
Replacing $\varphi_i$ by  $-k_n \big (A U_{ni} -f(t_{ni})\big )$ in the second term on the right-hand side,
we obtain
\[\begin{split}
(\varphi_i, AU_{ni})&=-k_n (AU_{ni}, AU_{ni})+k _n (f(t_{ni}), AU_{ni})\\
&=-k_n | AU_{ni}|^2+k(f(t_{ni}), AU_{ni});
\end{split}\]
thus, \eqref{eq:Bstab3} yields
\begin{equation}
\label{eq:Bstab4}
|A^{1/2}U_{n+1}|^2+2 k_n\sum _{i=1}^q b_i|AU_{ni}|^2
\leqslant |A^{1/2}U_n|^2+2 k_n\sum _{i=1}^q b_i(f(t_{ni}), AU_{ni}). 
\end{equation}
Using here the binomial identity
\[2(f(t_{ni}), AU_{ni})=-|AU_{ni}-f(t_{ni})|^2+|AU_{ni}|^2+|f(t_{ni})|^2=-|\widehat U'(t_{ni})|^2+|AU_{ni}|^2+|f(t_{ni})|^2,\]
we infer that
\begin{equation}
\label{eq:Bstab5} 
|A^{1/2}U_{n+1}|^2+ k_n\sum _{i=1}^q b_i|U'(t_{ni})|^2+ k_n\sum _{i=1}^q b_i|AU_{ni}|^2
\leqslant |A^{1/2}U_n|^2+ k_n \sum _{i=1}^q b_i|f(t_{ni})|^2. 
\end{equation}
Summing here over $n$ from $n=0$ to $n=m-1\leqslant N-1,$ we obtain the maximal regularity 
estimate
\begin{equation}
\label{eq:Bstab6} 
\begin{split}
|A^{1/2}U_m|^2&+  \sum _{n=0}^{m-1}k_n \sum _{i=1}^q b_i|\widehat U'(t_{ni})|^2+  \sum _{n=0}^{m-1}k_n \sum _{i=1}^q b_i|AU_{ni}|^2\\
&\leqslant |A^{1/2}U_0|^2+  \sum _{n=0}^{m-1}k_n \sum _{i=1}^q b_i|f(t_{ni})|^2,
\end{split} 
\end{equation}
$m=1,\dotsc,N,$ a discrete analogue of  \eqref{max-reg}.
Notice that \eqref{eq:Bstab6} reduces to  \eqref{eq:Ga-Ra6} for the Gauss and Radau IIA methods.

\section{General evolution problems and numerical results}\label{Se:5}

\subsection{Problem Setup}
\label{sec:orga758136}
In this section we present numerical results for the Runge--Kutta PINNs for both linear parabolic and wave equations. 
We would like to demonstrate that the resulting methods work as expected and in addition preserve the qualitative behavior of Runge--Kutta methods. 
We begin with the following general setup of the problem. Let
\(u:\varOmega\times (0,T] \rightarrow \mathbb R^M\), where \(T>0\) and \(u=u(x,t)\) is a vector-valued
function with \(M\) components. Let
\(A\) be a differential operator acting on \(u\) which involves
spatial derivatives. Our general initial value problem can be written as follows:
\begin{align}
  u_t  + A u &= f(x,t),\ (x,t)\in \varOmega \times [0,T],\\
u(0,x)&=u_0,\ x \in \varOmega, 
\end{align}
with additional boundary conditions for $t\in  [0,T].$ In the numerical experiments below the boundary conditions are either Neumann
 (for the heat equation) or Dirichlet (for the wave equation). The formulation of the methods can be directly extended to nonlinear evolution equations 
 with the obvious modifications in the loss functionals; see \cite{AMN2}.

\subsection{Collocation Runge--Kutta Formulation}
\label{sec:org18ce8ca}

Let the collocation points on \([t_n,t_{n+1}]\) be defined by \(t_{ni} = t_n + c_ik_n,\ i=1,\dotsc,q\), where \(0\leqslant c_1 < c_2<\dotsb< c_q \leqslant 1\),
and \(k_n=t_{n+1}-t_n\) is the width of each interval.

Furthermore, we let \(0=\tilde c_0 < \tilde c_1 < \dotsb < \tilde c_q=1\) be  auxiliary points as introduced in  Section \ref{SSe:2.4.2}, and
\begin{equation}
\hat u(x,t) = \hat I_q u(x,t) = \sum_{i=0}^q \tilde \ell_{ni}(t) u(x,\tilde
t_{ni}) = \sum_{i=0}^q \tilde \ell_i \left( \frac{t-t_n}{k_n}
  \right)u(x,\tilde t_{ni}),
\end{equation}
where \(\tilde t_{ni} = t_n + \tilde c_i k_n\). As it will become clear in Section 5.4, in the case of Radau and Lobatto methods, the points 
\(0=\tilde c_0 < \tilde c_1 < \dotsb < \tilde c_q=1\) include all the collocation points $c_1,\dotsc, c_q$.

Furthermore we need to interpolate at the collocation points the function $\hat u(x,t) , $ i.e., 
\begin{equation}
I_{q-1} L \hat u(x,t) = \sum_{j=1}^q \ell_{j}\left(\frac{t-t_n}{k_n}\right) \sum_{i=0}^q
\tilde \ell_{i}\left(\frac{t_{nj}-t_n}{k_n}\right) Lu(x,\tilde t_{ni}).
\end{equation}

Consider a fixed interval \([t_n,t_{n+1}]\). Recall that in the loss, one has to evaluate integrals of 
\begin{equation}
  \hat u_t(x,t)+ I_{q-1}  L \hat u(x,t) =: \zeta (x, t),  \, \quad\ t\in [t_n, t_{n+1}],\quad x\in \varOmega.
\end{equation}
At the colocation points \(\{t_{nj}\}_{j=1}^q\) it is straightforward to express the time derivative of \(\hat u\) and
the interpolant \(I_{q-1}  L \hat u(x,t)\) as follows:
\begin{equation}
\hat u_t(x,t_{nj}) = k_n^{-1} \sum_{i=0}^q \tilde \ell^\prime_i(c_j)u(x,\tilde
t_{ni}),\quad I_{q-1}L\hat u(x,t_{nj}) = \sum_{i=0}^q \tilde \ell_i(c_j) L u (x,\tilde t_{ni}).  
\end{equation}


Within the subinterval \([t_n, t_{n+1}]\), we have
\begin{equation}\label{formula_zeta_nj}
\begin{split}	
 \zeta (x, t_{nj})= & \hat u_t(x,t_{nj})+ I_{q-1}  L \hat u(x,t_{nj}) \\
  = & \sum_{i=0}^q \big(k_n^{-1} \tilde \ell_i^\prime(c_j) u(x,\tilde t_{nj}) -
  \tilde \ell_i(c_j) Lu(x,\tilde t_{nj})\big),\quad x \in \varOmega, 
\end{split}
\end{equation}
and since $ \hat u_t(x,t)+ I_{q-1}  L \hat u(x,t) $ is a polynomial of degree $q-1$
\begin{equation}\label{formula_zeta}
 \zeta (x, t)=  \hat u_t(x,t)+ I_{q-1}  L \hat u(x,t)  =  \sum_{j=1}^q \ell_{j}\left(\frac{t-t_n}{k_n}\right) \zeta (x, t_{nj}) ,\quad x \in \varOmega .
\end{equation}

\subsection{Artificial Neural Network Representation}
\label{sec:org75c7da6}
We approximate the solution of our problem within the neural network function space.
The objective is to find \(\theta^\star\in \Theta\) such that the \(m\)-th output \(U_m\) of the neural
network approximates the target function \(u_m(\cdot,\tilde t_{ni})\). Specifically, we require 
that the neural network output satisfies the following approximation
\[U_m(x,\tilde t_{ni};\theta^\star)\approx u_{m}(x,\tilde t_{ni}),\quad n=0,\dotsc,N,\quad
  i=1,\dotsc,q,\quad x \in \varOmega.\]
We denote the time-space variable as \(y = (x,t)\in \mathbb{R}^{d+1}\). We employ
the deep residual ANN architecture proposed by Sirignano and
Spiliopoulos \cite{SSpiliopoulos:2018}.
More specifically, for a time-space input \(y\), we define: 
\begin{align*} 
S^0 &= \tanh(W^{\mathrm{in}} y + b^{\mathrm{in}}),  \\
\quad \text{ DGM layer}\\
|\quad G^\ell &= \tanh(V^{g,\ell} y + W^{g,\ell} S^{\ell-1} + b^{g,\ell}),\ \ell = 1,\dotsc, L \\
|\quad  
Z^\ell &= \tanh(V^{z,\ell} y + W^{z,\ell} S^{\ell-1} + b^{z,\ell}),\ \ell = 1,\dotsc, L \\
|\quad R^\ell &= \tanh(V^{r,\ell} y + W^{r,\ell} S^{\ell-1} + b^{r,\ell}),\ \ell = 1,\dotsc, L \\
|\quad\! H^\ell &= \tanh(V^{h,\ell} y + W^{h,\ell}( S^{\ell-1} \odot R^{\ell}) + b^{h,\ell}),\ \ell = 1,\dotsc, L \\
 \lfloor\quad S^{\ell} &= (1-G^{\ell})\odot H^{\ell} + Z^{\ell}\odot S^{\ell-1},\ \ell = 1,\dotsc, L\\
U(y;\theta) & = W^{\mathrm{out}}S^L + b^{\mathrm{out}},
\end{align*}
where \(L\) denotes the number of hidden layers and \(\odot\) represents the Hadamard product. 
The trainable parameters \(\theta\) of the model are:
\begin{equation}
  \theta = \{W^{\mathrm{in}},b^{\mathrm{in}}, (V^{\star,\ell},W^{\star,\ell},b^{\star,\ell})_{\ell=1,\dotsc,L}^{\star\in\{g,z,r,h\}},W^{\mathrm{out}},b^{\mathrm{out}}\}.
\end{equation}

To train the neural network, we compute discrete approximations of
the cost functional based on samples from \(\varOmega\). For simplicity we consider $f=0;$ the modifications being obvious otherwise. 
 The discrete cost functional is computed as follows:
\begin{equation}
\mathcal C_\varOmega[\theta] = \frac{\mathrm{Vol}(\varOmega)}{R}\sum_{m=1}^M\sum_{r=1}^R\sum_{n=1}^N \int _ {J_n}    \zeta_m (x_r, t)^2\, \d t,
\end{equation}
where $\zeta_m $ denotes the \(m\)-th component of \(\zeta\) which is given by \eqref{formula_zeta} and
the set of points \(\{x_r\}_{r=1}^R\) is generated using Sobol's
low-discrepancy sequences; see, e.g., \cite{Caflisch_review_1998}.  This sampling approach is more
efficient for achieving a uniform coverage of the space in higher--dimensonal settings.
The integral in time is computed exactly by applying integration rules which are exact for  polynomials of degree $2q-2.$
In the case where we have Gauss or Radau collocation points this integral is just 
\begin{equation}
 \int _ {J_n}    \zeta (x_r, t)  ^2\, \d t 
 =  k_n \sum_{j=1}^q w_j  \zeta (x_r, t_{nj}) ^2,
\end{equation}
 where $ w_j= \int_0^1 \ell _j (\tau)\, \d \tau\, ,$ and $\zeta (x, t_{nj}) $ is given by \eqref{formula_zeta_nj}.

The following are the discrete cost functionals for the initial and boundary conditions.
\begin{equation}
\mathcal C_0[\theta] = \frac{\mathrm{Vol}(\varOmega)}{R}\sum_{m=1}^M\sum_{r=1}^R( u_m(0,x_r;\theta) - u_{m0}(x_r))^2
\end{equation}
\begin{equation}
\mathcal C_{\partial\varOmega_s}[\theta] = \frac{\mathrm{Vol}(\partial\varOmega_s)}{R}\sum_{m=1}^M\sum_{{r'}=1}^R
\sum_{n=1}^N (u_m(t_{n},x_{r'}) -u_{ms}(x_{r'} ))^2.
\end{equation}
where  $x_{r'}$ are Sobol points on the boundary ${\partial \varOmega_s}.$ In case we use other than Dirichlet boundary conditions, this term is modified accordingly.
The deep learning approximation of the problem is characterized by the minimization of the sum of the  cost functionals  with respect to model parameters \(\theta\),
\begin{equation}
\theta^\star \leftarrow \min_{\theta\in \Theta} (\mathcal C_\varOmega + \mathcal C_0 + \sum_s C_{\partial \varOmega_s})[\theta].
\end{equation}

\subsection{Applications}
\label{sec:org576a064}
In this section we will apply our Runge--Kutta ANN schemes to heat diffusion
and wave propagation initial value problems.

We will apply four alternative time--sampling approaches: three based on
collocation Runge-Kutta methods and, for comparison, a uniform
time--sampling scheme.
Specifically, we employ the following schemes:
\begin{itemize}
\item Gauss: \((c_1,c_2,c_3)=\left(0.5-\sqrt{15}/10,\ 0.5,\
    0.5+\sqrt{15}/10\right)\).
\item Lobatto IIIA: \((c_1,c_2,c_3)=(0,\  0.5,\  1)\).
\item Radau IIA: \((c_1,c_2,c_3) = \left((4-\sqrt{6})/10,\  (4+\sqrt{6})/10, 1\right)\).
\item Uniform Sampling.
\end{itemize}
The auxiliary modes for these collocation schemes are chosen as follows:
\begin{itemize}
\item Gauss and Lobatto IIIA: 
\((\tilde c_0,\tilde c_1, \tilde c_2, \tilde c_3) = (0,0.25,0.5,1)\),
\item Radau IIA: \((\tilde c_0,\tilde c_1, \tilde c_2, \tilde c_3) = \left(0,(4-\sqrt{6})/10,\  (4+\sqrt{6})/10, 1\right)\).
\end{itemize}
Notice that, in the case of  Radau and Lobatto methods, the points 
\(0=\tilde c_0 < \tilde c_1 <\tilde c_2 < \tilde c_3=1\) include  the collocation points $c_1,\dotsc, c_3$.
To ensure a fair comparison, we use four times the value of \(n\) for the uniform sampling case compared 
to the Runge--Kutta collocation schemes with \(q=3\).

In both of our applications, we utilized artificial neural networks as
described earlier with four hidden layers, each consisting of \(20\)
nodes. The training process employed the Adam optimizer with a
learning rate of \(3\cdot10^{-4}\). The models were trained for \(20000\)
epochs for the heat equation and \(100000\) epochs for the wave equation. For the time discretization, we used \(40\) nodes, while the
spatial grid was generated using Sobol's sequences with 256 points per
epoch. These hyperparameters were carefully selected to balance computational efficiency with the accuracy of the neural network approximations of the solutions.

\subsubsection{Heat Equation}
\label{sec:org780d570}

We consider an initial value diffusion problem on \(\varOmega \in
(0,1)^2\) subject to Neumann boundary conditions:
\begin{align}
    u_{t} - k (u_{xx} + u_{yy}) &= 0, \quad t \in [0,1],\ (x,y) \in \varOmega,\ k=0.02, \\
    u(0, x, y) &= \left(0.5 + 0.5\cos\left(10\pi \sqrt{(x - 0.6)^2 + (y - 0.7)^2}\right)\right) \chi_{D}(x, y), \\
    \frac{\partial u}{\partial n}&= 0, \quad \text{on }  \partial \varOmega \times (0,1),
\end{align}
where \(\chi_D\) is the characteristic function of the disk $D,$ 
\[D:=\{(x,y)\in \R^2: (x - 0.6)^2 + (y - 0.7)^2 < 0.01\}.\]
%
Note that the initial value is nonzero inside the disk $D$ and zero outside $D.$

Fig.\ \ref{fig:orgf92d926} shows the absolute errors at the final time $t=1$ comparing the different schemes.

\begin{figure}[htbp]
\centering
\includegraphics[width=1.\linewidth]{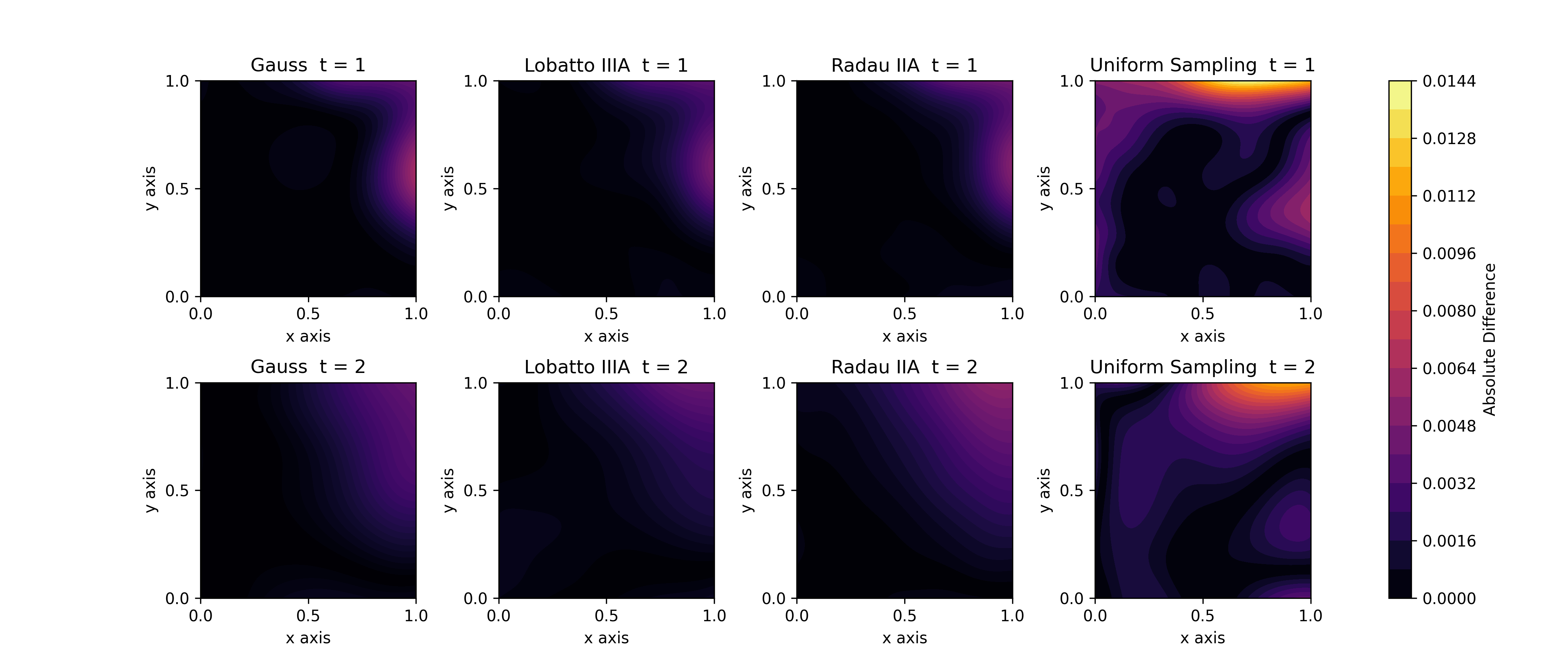}
\caption{\label{fig:orgf92d926}Absolute misfit across various  schemes}
\end{figure}
As is well known, the total heat, in the system remains conserved,  i.e., 
\[\int _{\varOmega} u (x, t )\, \d x = \int _{\varOmega} u (x, 0 )\, \d x\, .\]
   The
conservation arises form the Neumann boundary conditions, which implies that there is no heat flow across the boundaries.
In Fig.\ \ref{fig:org75fd584}, we compare the various schemes  with respect to heat conservation.

\begin{figure}[htbp]
\centering
\includegraphics[width=1.\linewidth]{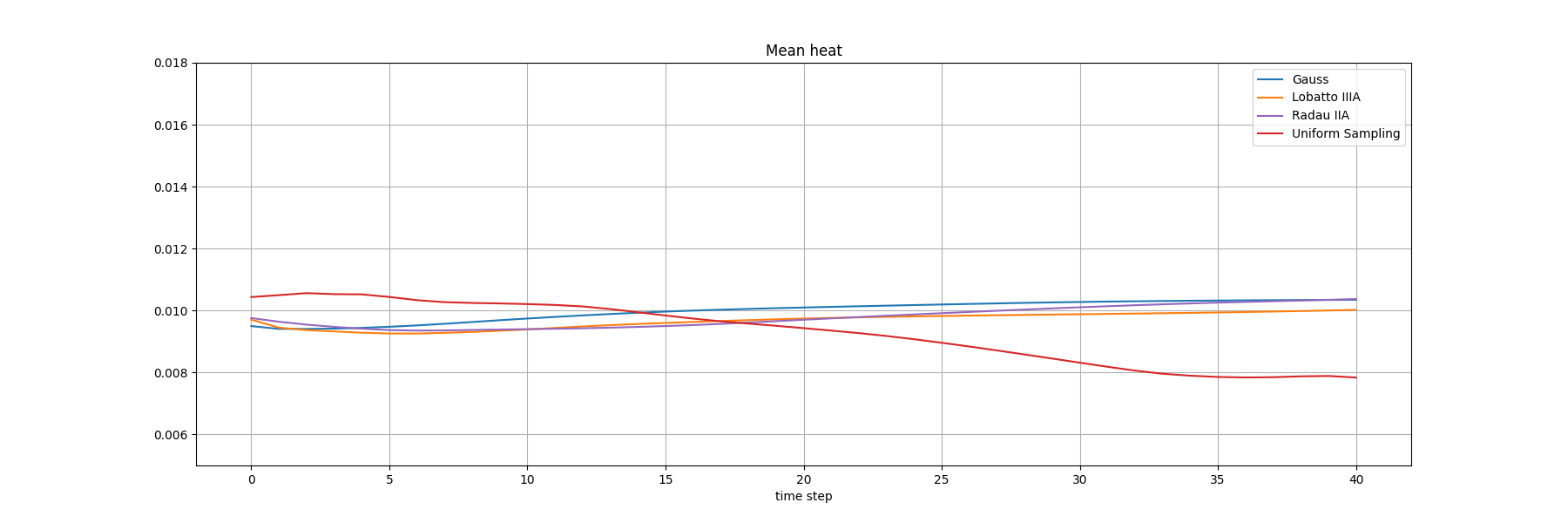}
\caption{\label{fig:org75fd584}Heat conservation (plot of $H(t)=\int _{\varOmega} u (x, t )\, \d x $)  across
  various  schemes.}
\end{figure}

\subsubsection{Heat Equation with a Discontinuous Initial Value}
We replace the initial value from the previous test case with a discontinuous one. Specifically, we consider
\begin{equation}
u(0, x, y) = \chi_D(x, y), 
\end{equation}
 the characteristic function of the disk $D,$ 
\[D:=\{(x,y)\in \R^2: (x - 0.6)^2 + (y - 0.7)^2 < 0.01\}.\]
%
Fig.\ \ref{fig:ex2sol} shows the absolute errors where comparing the collocation schemes.
\begin{figure}[htbp]
\centering
\includegraphics[width=1.\linewidth]{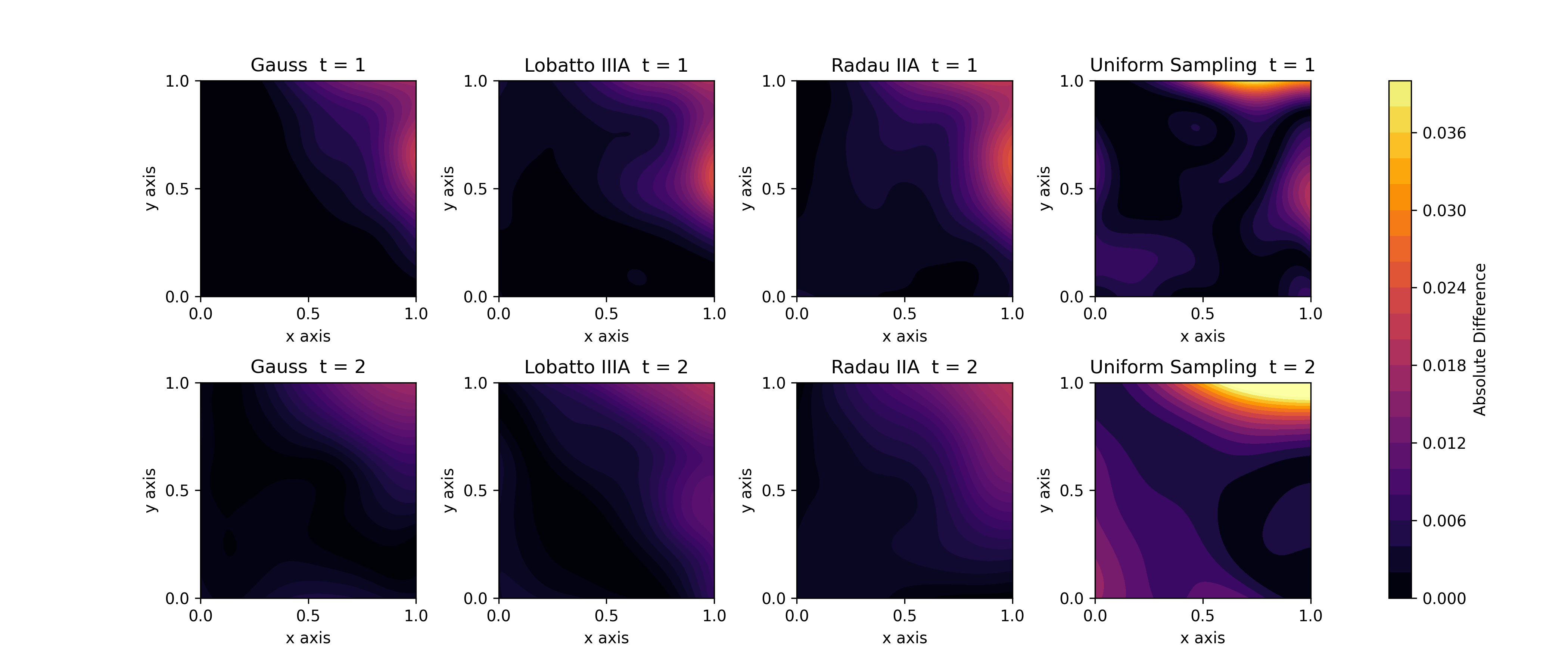}
\caption{\label{fig:ex2sol}Absolute misfit across various iterative schemes}
\end{figure}
The heat equation exhibits a notable smoothing property, which is not always maintained by various time 
discretization methods; see, e.g., \cite{Luskin_Ran_CN_1982}, \cite{Thomee_book_2006}. 
The smoothing behavior of time discretizations is a quite subtle topic, as it can influence the performance 
of methods when both diffusion and transport phenomena are present, particularly in the context of nonlinear 
PDEs and Navier--Stokes equations; see \cite{Glow_theta1987}. Although smoothing eventually occurs primarily 
due to the diffusion induced by the Quasi-Monte Carlo method used for spatial discretization, it is worth noting, 
as shown in Fig.\ \ref{smoothing}, that the behavior of various discretization methods aligns with the known 
predictions of the corresponding Runge--Kutta schemes. In Fig.\ \ref{fig:ex2energy}, we compare the various 
schemes with respect to heat conservation.
\begin{figure}[htbp]
\centering
\includegraphics[width=1.\linewidth]{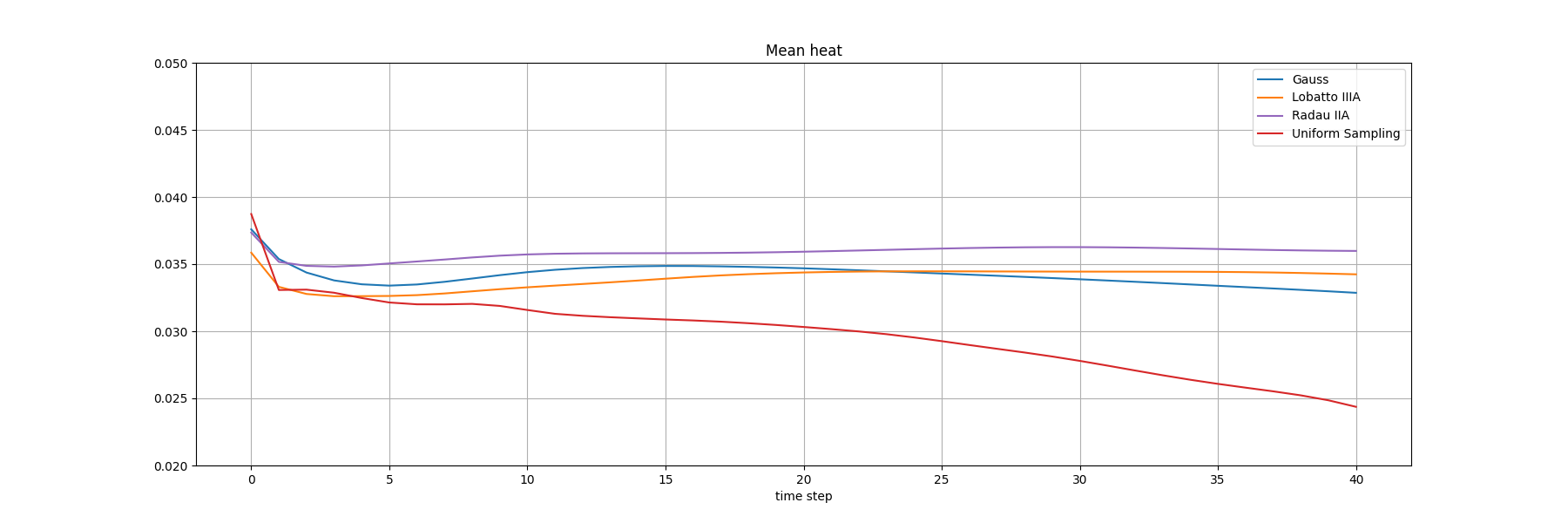}
\caption{\label{fig:ex2energy}Heat conservation across
  various iterative schemes.}
\end{figure}
\begin{figure}[htbp]
\centering
\includegraphics[width=1.\linewidth]{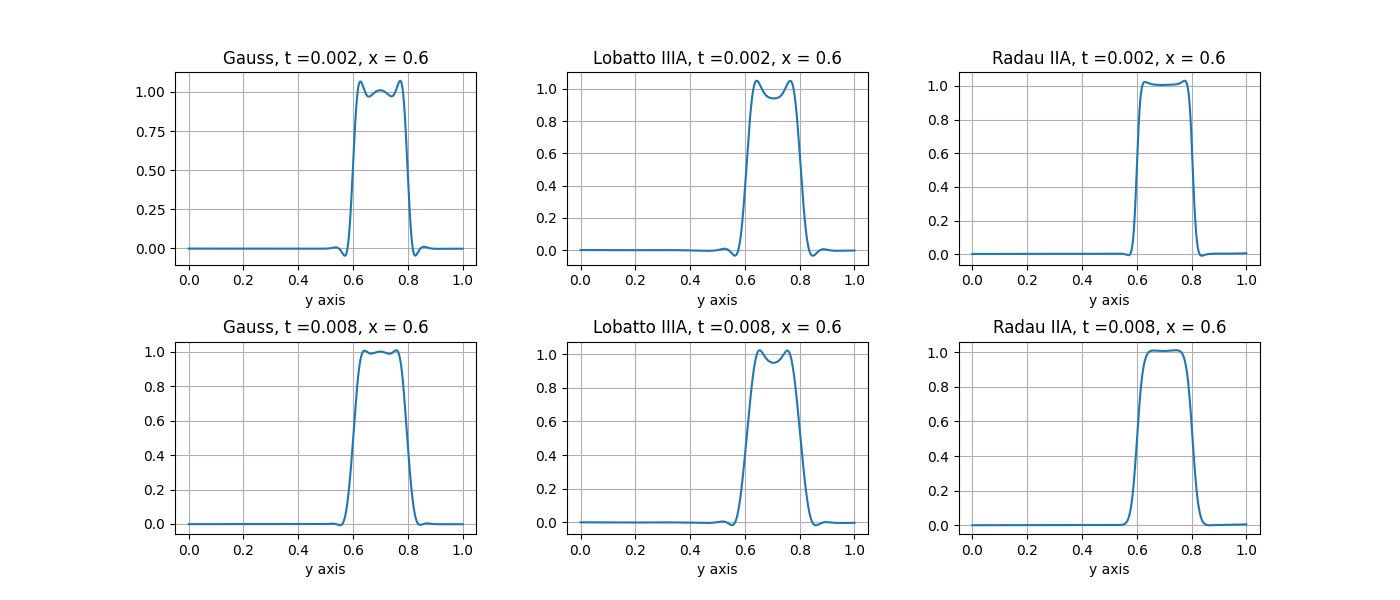}
\caption{\label{smoothing} Smoothing effect for the heat equation. As expected, Gauss and Lobatto methods have 
oscillating behavior close to initial times for discontinuous data. The full smoothing effect of Radau methods is evident.}
\end{figure}

\subsubsection{Wave Equation}\label{sec:orgc685aab}
As a third example, we consider an initial value wave propagation
problem on \(\varOmega:=(0,1)^2\) with homogeneous Dirichlet boundary conditions:
\begin{align}
    u_{tt} - c^2 (u_{xx} + u_{yy}) &= 0, \quad t \in [0,1],\ (x,y)\in \varOmega,\ c=0.5, \\
    u(0, x, y) &= \left(0.5 + 0.5\cos\left(4\pi \sqrt{(x - 0.3)^2 + (y - 0.5)^2}\right)\right) \chi_D(x, y), \\
    u_t(0, x, y) &= 0,\quad (x,y)\in \varOmega,\\
    u &= 0, \quad \text{on } \partial \varOmega \times (0,1)
\end{align}
where  $\chi_D$ is the characteristic function of the disk $D,$
\[D:=\{(x,y)\in \R^2: (x - 0.3)^2 + (y - 0.5)^2 \leqslant 0.25^2\}.\]
%
As in the previous application, the support of the initial field is
a closed disk.

To reformulate this as a system of first-order equations, we introduce
an auxiliary variable \(v\), representing the velocity, \(v:=u_t\).
The corresponding system of first-order equations is given as follows:
\begin{equation}
\begin{pmatrix} u \\ v \end{pmatrix}_t + \begin{pmatrix} 0 & -I \\ -c^2\varDelta & 0 \end{pmatrix} 
\begin{pmatrix} u \\ v \end{pmatrix} = \begin{pmatrix} 0 \\ 0\end{pmatrix},
\end{equation}
where \(I\) stands for the identity operator.

For each sampling approach,
Fig.\ \ref{fig:orgc63e982} illustrates the absolute difference between
the solution estimate
at time \(t=1\) and the corresponding estimate
delivered using the method of separation of variables, with 8 terms retained in the summation.
\begin{figure}[htbp]
\centering
\includegraphics[width=\textwidth]{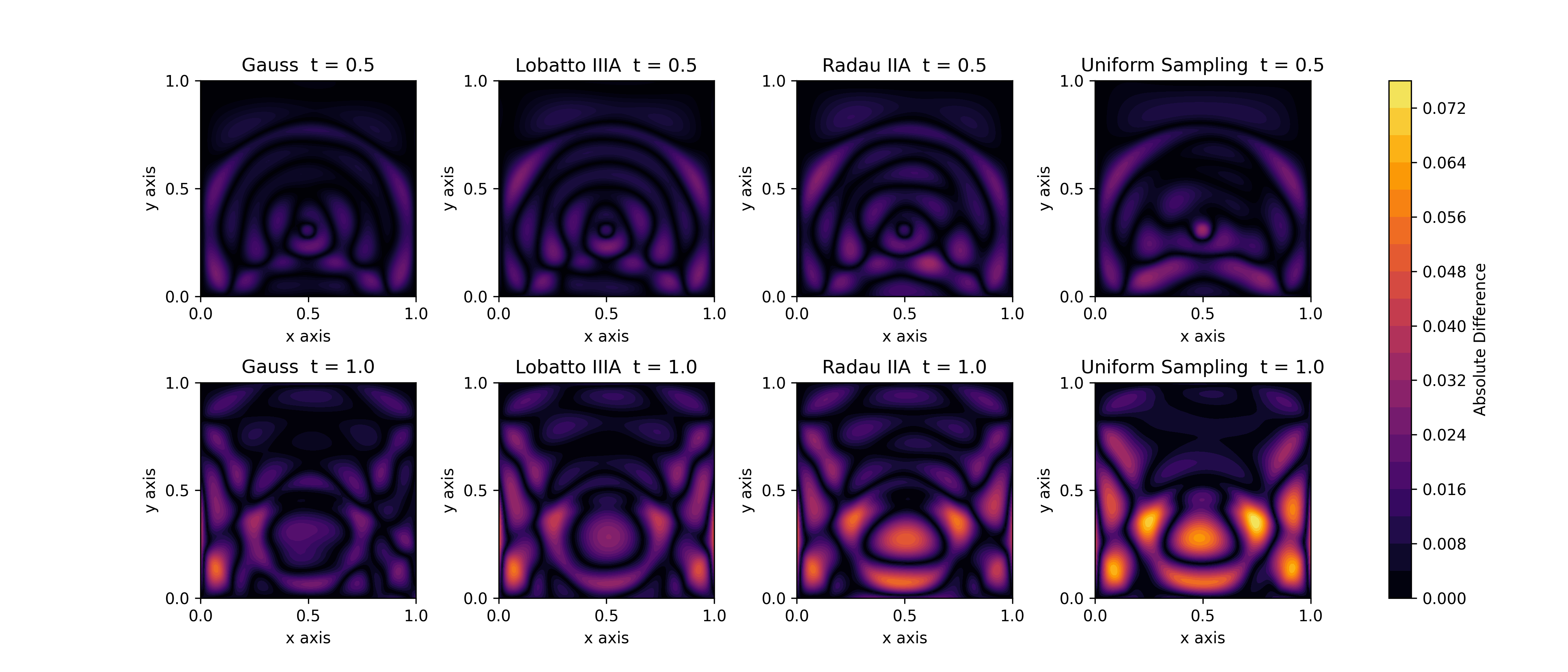}
\caption{\label{fig:orgc63e982}Absolute misfit for the various iterative schemes}
\end{figure}
The total energy of the system is given by 
\begin{equation}
\frac{1}{2}\left \|u_t(t,\cdot)\right \|^2 + \frac{1}{2}c^2 \left \| \nabla u(t,\cdot)\right \|^2 
\end{equation}
with $\|\cdot\|$ the $L^2(\varOmega)$-norm.

Since the total energy is conserved over time, Fig.\ \ref{fig:org7cafb74} presents a comparison of the system's 
energy at each time step for the  schemes under evaluation. While all methods exhibit some diffusion due to 
the Quasi-Monte Carlo method used for spatial discretization, it is evident that the Gauss and Lobatto methods 
perform as expected. These methods stand out as the optimal choice when energy conservation and high accuracy 
are priorities. Interestingly, the full sampling method shows the highest level of diffusive behavior.

%
\begin{figure}[htbp]
\centering
\includegraphics[width=1.1\linewidth]{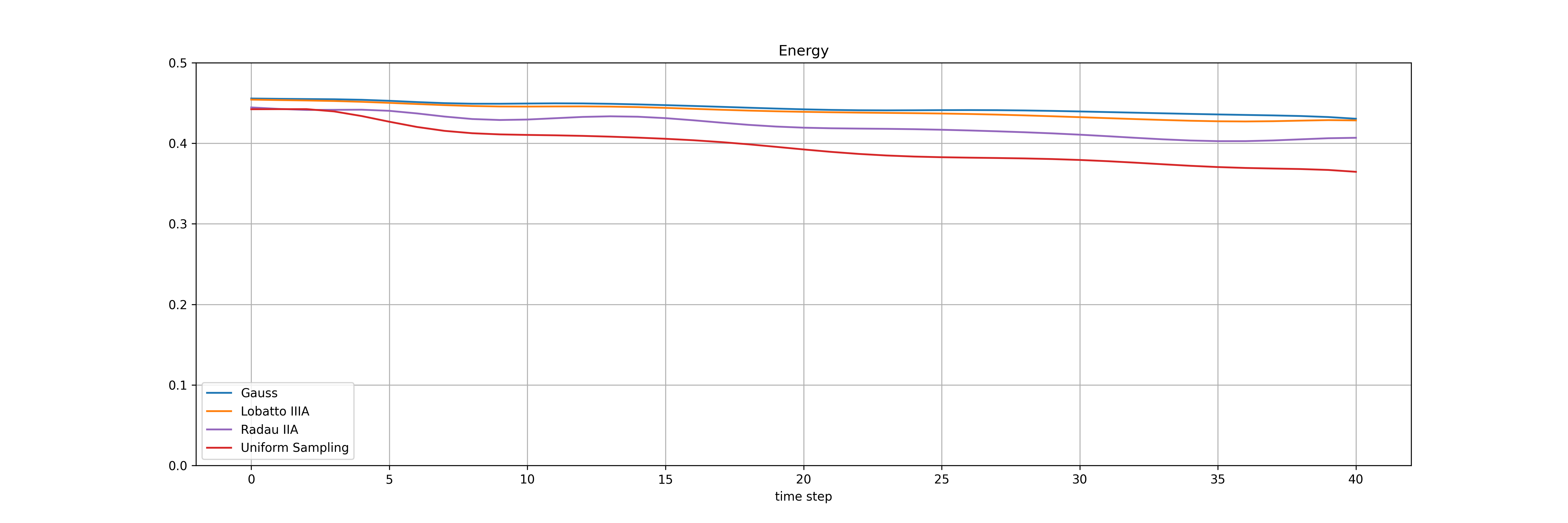}
\caption{\label{fig:org7cafb74}Energy conservation across the various
   schemes. Gauss and Lobatto methods have superior conservation properties. }
\end{figure}

\printbibliography

\end{document}